\documentclass{article}

\usepackage{graphicx} 
\usepackage{amsmath,amsfonts,amssymb,mathtools, amsthm}
\usepackage{xcolor}
\usepackage{enumitem}
\usepackage{wasysym}
\usepackage{hyperref}
\usepackage{cleveref}
\usepackage{rotating} 
\usepackage{subcaption}

\usepackage[utf8]{inputenc}
\usepackage[T1]{fontenc}

\def\Ltwo{L_2}
\def\P{\mathbb{P}}

\def\x{x}
\def\y{y}
\def\v{\mathbf{v}}
\def\Ox{\Omega_x}
\def\Oy{\Omega_y}
\def\divx{\triangledown_\x \cdot}

\newcommand\expec{\mathbb{E}}
\newcommand{\norma}[3]{\|{#1}\|_{#2(#3)}}
\newcommand\A{\mathbf{A}}
\newcommand\Seig{\mathbf{S}}
\newcommand\Leig{\mathbf{\Lambda}}
\newcommand\bq{\mathbf{q}}
\newcommand{\gpcsum}[3]{\sum_{n=0}^{#1}\, \widehat{#2}_n({#3},t)P_n(y)}
\newcommand{\uhat}[2]{\widehat{#1}_{#2}}

\newcommand{\pder}[3]{\frac{\partial^{#1}{#2}}{\partial {#3}}}

\newtheorem{thm}{Theorem}

\usepackage{enumitem}
\usepackage[utf8]{inputenc}

\newcommand{\ignore}[1]{}
\usepackage{comment}
\usepackage{todonotes}
\newlist{todolist}{itemize}{2}
\setlist[todolist]{label=$\square$}

\newcommand{\KH}[4]{K_{#1}^{(#2,#3)}\left(#4\right)}
\newcommand{\FKH}[4]{\hat{K}_{#1}^{(#2,#3)}\left(#4\right)}
\newcommand{\FullK}[3]{\sum_{\gamma = 0}^{#1}\, c_\gamma #2(#3-x_\gamma)}
\newcommand{\BS}[1]{B^{(#1)}}
\DeclareMathOperator{\sinc}{sinc}

\title{SIAC Accuracy Enhancement of Stochastic Galerkin Solutions for Wave Equations with Uncertain Coefficients}
\author{Andr\'es Galindo-Olarte\thanks{Oden Institute for Computational Sciences and Engineering, University of Texas at Austin, Austin, TX 78712, USA, (afgalindo@utexas.edu).}\and Jennifer Ryan\thanks{Department of Mathematics, KTH Royal Institute of Technology, 114 28 Stockholm, Sweden, (jryan@kth.se).}}
\date{\today}

\begin{document}

\maketitle
\begin{abstract}
This article establishes the usefulness of the Smoothness-Increasing Accuracy-Increasing (SIAC) filter for reducing the errors in the mean and variance for a wave equation with uncertain coefficients solved via generalized polynomial chaos (gPC) whose coefficients are approximated using discontinuous Galerkin (DG-gPC). Theoretical error estimates that utilize information in the negative-order norm are established.  While the gPC approximation leads to order of accuracy of  $m-1/2$ for a sufficiently smooth solution (smoothness of $m$ in random space), the approximated coefficients solved via DG improves from order $k+1$ to $2k+1$ for a solution of smoothness $2k+2$ in physical space.  Our numerical examples verify the performance of the filter for improving the quality of the approximation and reducing the numerical error and significantly eliminating the noise from the spatial approximation of the mean and variance. Further, we illustrate how the errors are effected by both the choice of smoothness of the kernel and number of function translates in the kernel. Hence, this article opens the applicability of SIAC filters to other hyperbolic problems with uncertainty, and other stochastic equations. 
\end{abstract}

\section{Introduction}
\label{secn:introduction}

Generalized Polynomial Chaos (gPC) has long been established as a useful tool for measuring uncertainty.  It allows for constructing useful insight into complex systems by taking a linear combination of many different snapshops of data that depend on a distribution of a random variable  \cite{xiu2002wiener,xiu2003modeling}.  In \cite{gottlieb2008galerkin}, gPC utilizing Legendre polynomials was established with the weights of the coefficients being defined by a mathematical model of some quantity of interest. When the equation defining the coefficients is nonlinear, it is often necessary to approximate the coefficients.  Doing so through discontinuous Galerkin methods (DG) is one choice for choosing the approximation.  In this article, we demonstrate how the errors when approximating the coefficients can be reduced through the use of a moment-based filter -- the Smoothness-Increasing Accuracy-Increasing (SIAC) filter. This filter reduces the discretized noise that arises from approximating the coefficients and improves the order of accuracy for those coefficients from order $k+1$ to up to $2k+1$ for coefficients defined by hyperbolic conservation laws.  We theoretically and numerically prove the effectiveness of the SIAC filter for reducing the error and accelerating the convergence rate.

A multi-element gPC expansion was introduced in \cite{wan2006long}, where a collocation multi-element method was implemented for the wave equation and $2D$ flow past a circular cylinder.  However, it relies on an increasing number of coefficients during time evolution.  In \cite{Motamedetal2013} the coefficients are for a two-dimensional wave equation were approximated using either finite difference or finite elements in the physical space with a tensor-product collocation method based on orthogonal polynomials (Gauss points) in the probability space. A local discontinuous Galerkin method combined with a leap frog time-stepping scheme was shown to be effective for a two-dimensional wave equation with a spatial discontinuity \cite{chou2019energy}.  

Accuracy enhancement of the gPC expansion is not new.  Tang and Zhou utilized a collocation approach together with a domain decomposition idea \cite{tang2010convergence}. In their work, they used a local reconstruction combined with a Bayes reconstruction of the means to obtain accuracy enhancement.  
In \cite{kusch2020filtered}, oscillations were adaptively dampened in regions of uncertainty.  The authors also introduced a new filter based on Lasso regression [24]. The resulting filter depends on the gPC coefficients and sets small-magnitude and high-order coefficients to zero, yielding sparsity in the filtered coefficients.  Indeed, filters are often introduced to aid in addressing the curse of dimensionality problem \cite{chou2019energy,cci̇louglu2023stochastic,ccilouglu2022stochastic}.

In this article, instead of utilizing Bayes rule for approximating a non-local mean, we take the local approximation to the coefficients using DG and utilize a compactly-supported moment-based limiter to reconstruct a non-local average, specifically, the SIAC filter.  The SIAC filter aids in accelerating the decay of the approximated coefficients in Fourier space.  DG approximations are known to have higher-order information in the dissipation and dispersion errors. The filter is able to draw out this higher order information from Fourier space so that it can be seen in physical space.   Further, as shown in \cite{Adjerid1,lowerie}, there are superconvergent points that give the error a regular pattern of oscillations.  This can but measured through the negative-order Sobolov norm, which can be used to bound the $L^2-$error \cite{bramble1977}.  In this article, theoretical insight into the effectiveness of the filter, including negative order norm estimates, are shown together with a numerical demonstration of its effectiveness.  Additionally, the filter consists of three filtering parameters that allow for choosing the amount of dissipation, allowable accuracy, and scale for the pattern of oscillations. The affect of utilizing the different filtering parameters for dissipation and accuracy is demonstrated numerically.  These results will be useful in uncertainty quantification as they enable a better approximation to statistical quantities of interest.  

The article outline is as follows: In Section \Cref{secn:background} we present the notation that will be utilized throughout this article as well as the basics of the gPC method, including the theoretical convergence.  In Section \Cref{scn:dg_discretization}, we review the application of the DG method to a system of equations for the polynomial chaos coefficients.  We then discuss the SIAC filter in Section \Cref{secn:siac} and proof of higher order accuracy of statistical quantities of interest in the negative-order Sobolev norm.  We also prove both theoretically and numerically (Section \Cref{secn:numerical_experiments}) the effectiveness of SIAC in drawing the high-order information out.

\section{Background}
 \label{secn:background}
 

In order to understand how accuracy enhancing post-processing can aid in approximating stochastic data, we begin by considering the case where the coefficients are defined by a linear hyperbolic equation:
\begin{align}
    \label{eqn:model}
    \pder{}{u}{t} = & c(y)\pder{}{u}{x},\qquad (\x,t;y) \in \Ox \times (0,T] \times \Oy\\ 
    u(x,0;y) = & u_0(x;y)\notag
\end{align}
Here, $u = u(x,t;y),$ with $x$ and $t$ representing (cartesian) space and time and $y\in\Omega_y=(-1,1)$ in properly defined complete random space with probability distribution $\rho(y)$. In this article we will only consider periodic boundary conditions in physical space $x$. The boundary conditions for the original scalar equation, \eqref{eqn:model}, depend on the particular realization of the random variable $y$ and hence make enforcing the inflow-outflow boundary conditions challenging For a discussion on enforcing other types of boundary conditions, see \cite{gottlieb2008galerkin}. We will utilize some typical notation, defined in the following section.  

\subsection{Notation}
\label{secn:notation}

In this section, common notation that is used throughout this article is introduced.  We start by defining the standard $\Ltwo-$inner product for functions, $f,\, g\in L^2(\Omega_x)$,

\begin{equation}
    (f,g)_{\Ox} = \int_{\Ox}\, f(\x)g(\x)\,d\x.
    \label{eqn:innerprod}
\end{equation}
We will also use $\norma{\cdot}{\Ltwo}{\Ox}$ to denote the standard $\Ltwo$-norm:
\begin{equation}
    \norma{f}{\Ltwo}{\Ox}^2=\int_{\Ox}|f(\x)|^2\,d\x
    \label{eqn:l2_norm}
\end{equation}
Similarly, $\norma{\cdot}{H^\ell}{\Ox}$ denotes the usual Sobolev norm in cartesian space, $\x$:
\begin{equation}
    \norma{f}{H^\ell}{\Ox}^2=\int_{\Ox} \sum_{\alpha=0}^{\ell}\left[\frac{\partial^{\alpha}f(\x)}{\partial x^{\alpha}}{}\right]^2\,d\x,
\label{eqn:sobolev_norm} 
\end{equation}
We define the negative order Sobolev norm on an open set $\Ox$ to be
\begin{equation}
\norma{f}{\mathcal{H}^{\ell}}{\Ox}=\sup_{\Phi\in\mathcal{C}_0^\infty(\Ox)}\, \frac{\left(f(x),\Phi(x)\right)_{\Ltwo(\Ox)}(T)}{\norma{\Phi}{\mathcal{H}^{\ell}}{{\Ox}}}.
\label{eqn:negative_order_norm}
\end{equation}
As pointed out in \cite{CLSS}, the negative norm error estimate gives us information on the oscillatory nature of the error. 

As it is the statistical quantities are of interest, we define the mean, or expectation, of a given function as
\begin{equation}
    \mu_{f} = \expec[f]=\int_{\Omega_y}f(\y)\rho(\y)\,d\y.
    \label{eqn:def_expectation}
\end{equation}
and the variance as 
\begin{equation}
    \sigma_{f}^2 = \expec[(f-\mu_f)^2]=\expec[f^2]-\expec[f]^2.
    \label{eqn:def_variance}
\end{equation}

In this article, we make use of the set of polynomials, $P_n$, related to the distribution of the random variable $\y$ that satisfy the following orthogonality relation:
\begin{equation}
    \expec[P_nP_m]=\int_{\Oy}P_n(\y) P_m(\y)\rho(\y)\,d\y=\delta_{n,m},\quad\forall n,m,
    \label{eqn:chaos_basis_def}
\end{equation}
where $\delta_{n,m}$ is the Kronecker delta function and $\expec[f]$ is the expectation of the function $f=P_nP_m$. Notice that $\rho(\y)$ helps to enforce normalization of the polynomials, $P_n.$  Hermite-Gaussian (the original polynomial-chaos expansion), Legendre-uniform, and Laguerre-Gamma, etc.,  are just a few examples the correspondence of polynomials and the distribution of the random variable $\y$, c.f., \cite{xiu2002wiener,xiu2003modeling}. As done in \cite{xiu2002wiener,xiu2003modeling}, we limit the discussion to the one-dimensional case of the random variable $y$ having a beta distribution on $\Oy = (-1,1)$ (upon proper scaling) and $\Omega_x=[0,L_x]$, where $L_x>0$.  For this particular case the generalized Polynomial Chaos (gPC) expansion will utilize the (normalized) Jacobi polynomials. 



%
%

\subsection{Polynomial chaos Galerkin approximation} 
\label{ssecn:polynomial_chaos}

We assume that $u(\x,t;y)$ is sufficiently smooth in the random variable $y$ and has a converging expansion of the form 
\begin{equation}
    u(\x,t,y)=\gpcsum{\infty}{u}{\x},
    \label{eqn:chaos_expansion}
\end{equation}
where the polynomials $P_n(y)$ are as discussed in Section \ref{secn:notation}.  Here we  assume that the coefficients $\uhat{u}{n}$ decay fast asymptotically, i.e., that there exists a $n_0>0$ such that for all $n>n_0$, the following condition is satisfied:
\begin{equation}
    \norma{\uhat{u}{n}(\cdot,t)}{H^1}{\Ox)}^2 \leq \frac{K}{N^{2m}}
    \label{eqn:decay_condition}
\end{equation}
where $N,\, m>0$ are constants. As in the case of spectral methods, $m$ can be interpreted as the smoothness in the random variable $y$. 

Employing a Galerkin projection for the exact solution, we obtain an infinite system of equations for the expansion coefficients, $\uhat{u}{n}({\x},t):$ 
\begin{equation}
    \pder{}{\uhat{u}{n}}{t}({\x},t)=\sum_{k=0}^{\infty}a_{n,k}\frac{\partial\uhat{u}{k}}{\partial x}({\x},t),\quad n\in \mathbb{N}_{\geq 0}.
    \label{eqn:coeff_infinite_system}
\end{equation}
with $a_{n,k}$ given by 
\begin{equation}
    a_{n,k}=\int_{\Oy}\, c(y)P_n(y)P_k(y)\rho(y)\,dy.
    \label{eqn:a_coeff_defn}
\end{equation}
Note that we can write this as 
\begin{equation}
    \pder{}{\uhat{u}{n}}{t}({\x},t) = \sum_{k=0}^{N}a_{n,k}\pder{}{\uhat{u}{k}}{x}({\x},t)+\sum_{k=N+1}^{\infty}a_{n,k}\pder{}{\uhat{u}{k}}{x}({\x},t),\quad n=0,\cdots, N.
    \label{eqn:coeff_infinite_system_truncated}
\end{equation}

In the gPC Galerkin method, an approximation to the exact solution via a finite term gPC expansion is utilized.  We define this expansion by $v(\x,t;y):$
\begin{equation}
    v(\x,t,y) = \gpcsum{N}{v}{\x}.
    \label{eqn:gPC_solution}
\end{equation}
The truncated model problem is then given by 
\begin{equation}
    \frac{\partial v}{\partial t}-c(y)\pder{}{v}{x} = 0
\end{equation}

Performing a Galerkin  projection onto  the subspace spanned by the gPC basis polynomials, $P_n,\, n=0,\dots, N,$ the following truncated system of equations is obtained:  
\begin{equation}
    \pder{}{\uhat{v}{n}}{t}=\sum_{k=0}^{N}a_{n,k}\pder{}{\uhat{v}{k}}{x},\quad n=0,\cdots, N,
    \label{eqn:gpc_method}
\end{equation}
here the $a_{n,k}$ are defined in \eqref{eqn:a_coeff_defn}. Denote by $\A$ the $(N+1)\times(N+1)$ matrix whose entries are $\{a_{n,k}\}_{0\leq n,k\leq N}$ and $\v=(\uhat{v}{0},\dots,\uhat{v}{N})^\top,$ a vector of length $(N+1)$. Then the system \eqref{eqn:gpc_method} can be cast as 
\begin{equation}
    \label{eqn:gpc_matrix_system}
    \pder{}{\v}{t}=\A\pder{}{\v}{x}.
\end{equation}
Note that from the definition in Equation  \eqref{eqn:a_coeff_defn} of the entries of $\A$, the system in Equation \eqref{eqn:gpc_matrix_system} is symmetric hyperbolic, $\A=\A^\top$.

In this article, we approximate the expansion coefficients in this matrix system utilizing the discontinuous Galerkin (DG) method, which will be discussed in Section \ref{scn:dg_discretization}.  In order to do utilize the DG method, it will be necessary to perform an eigenvalue decomposition of the matrix $\A.$ This is discussed in the next section.

\subsubsection{Eigenvalue decomposition for the gPC system}
\label{sec:eigen_boundary} 

For simplicity of implementation of the DG method that we will use for approximation of the coefficients as well as for implementation of the boundary conditions, an eigenvalue decomposition will be employed.   Since $\A$ is symmetric, there is an orthogonal matrix $\mathbf{S}^\top=\mathbf{S}^{-1}$ such that
\begin{equation}
    \Seig^{\top}\A\Seig=\Leig,\label{eqn:diagonal}
\end{equation}
where $\Leig$ is the diagonal matrix whose entries are the eigenvalues of $\A,$ i.e;
\begin{equation}
    \Leig=\mathrm{diag}(\lambda_0,\dots,\lambda_{n_{+}},\dots,\lambda_{(N+1)-n_{-}},\dots,\lambda_N).
\end{equation}
For simplicity, here we assume that the positive eigenvalues occupy indices $n=0,\dots,n_{+}$, and the negative eigenvalues occupy indices $n = (N+1)-n_{-},\dots,N,$ with $n_{+} + n_{-}\leq N$. The remaining indicies, $N-(n_{+} + n_{-}),\cdots, N-n_{-}$ eigenvalues, if they exist, are zero.   Note that we can then write
\begin{equation}
    \Leig= \Leig^+ + \Leig^-,
\end{equation}
where $\Leig^+,\, \Leig^-$ are $(N+1)\times(N+1)$ matrices containing the positive and negative eigenvalues, respectively:
$$\Leig^+ = \mathrm{diag}(\lambda_0,\dots,\lambda_{n_{+}},0,\dots,0),\qquad \Leig^- = \mathrm{diag}(0,\dots,0,\lambda_{(N+1)-n_-},\dots,\lambda_N).$$

Defining $\bq=(\hat{q}_0,\dots,\hat{q}_N)^{\top}=\Seig^{\top}\v$, i.e, 
\begin{equation}
    q_n(\x,t)=\sum_{k=0}^N s_{n,k}\uhat{v}{k}(\x,t),\label{eqn:q_k_def}
\end{equation}
where the $s_{n,k}$ are the entries of $\Seig$ leads to a decoupled system of equations: 
\begin{equation}
    \frac{\partial \bq}{\partial t}= \Leig\divx{\bq} = \Leig^-\pder{}{\bq}{x}+\Leig^+\pder{}{\bq}{x}.
     \label{eqn:q_system}
\end{equation}

 While in this article, we only discuss periodic boundary conditions, the eigenvalue decomposition above can better help us understand the implementation of the boundary conditions \cite{gottlieb2008galerkin}.  

 \subsection{Convergence in random space}\label{ssec:convergence}

 
As our results will also depend upon the convergence of the gPC expansion, we review the results from \cite{gottlieb2008galerkin}.  The convergence of the gPC expansion in the random variable $y$ are provided by the following theorem:

\begin{thm}
\label{thm:spectral_convergence}
    Consider the hyperbolic equation \eqref{eqn:model} where $y$ is a random variable with beta distribution in $\Oy = (-1,1)$. Let $u(x,t;y)$ be the solution of \eqref{eqn:model} whose exact gPC expansion is given by Equation \eqref{eqn:chaos_expansion}.  Denote the truncated gPC expansion by $v(x,t;y),$  which is defined by Equation \eqref{eqn:gPC_solution} and solved via the Galerkin system \eqref{eqn:gpc_matrix_system} with periodic boundary conditions. Then, for any finite time $t \in (0,T],$
    \begin{equation}
        \expec\left[\norma{u-v}{\Ltwo}{{\Ox}}^2\right](t)=\sum_{n=0}^N\, \norma{\uhat{u}{n}-\uhat{v}{n}}{\Ltwo}{{\Ox}}^2 + \sum_{n=N+1}^\infty\, \norma{\uhat{u}{n}}{\Ltwo}{{\Ox}}^2\leq t\frac{K}{N^{2m-1}}.
        \label{eqn:gPC_error_estimate}
    \end{equation}
    where $m$ is given in \Cref{eqn:decay_condition}.
\end{thm}
We note that this has linear growth in time.  Additionally, it relies on having a sufficient number of coefficients of sufficient smoothness.

\section{Approximation of coefficients via the discontinuous Galerkin method.}
\label{secn:numerical_scheme}

Now that we have introduced the basics of the gPC approximation for the random part of the solution $u(x,t;y)$,  we introduce the DG discretization for approximating the coefficients. 

\label{scn:dg_discretization}

Let $\{\mathcal{T}_h\}$ be a partition of our physical domain $\Omega_x=[a,b]$,
\begin{equation}
    \mathcal{T}_h=\{I_i=[x_{i-1/2},x_{i+1/2}],\,1\leq i\leq N_x\},
\end{equation}
and the mesh size $h$ relative to the partition is defined as
\begin{equation}
    h=\max_{1\leq i\leq N_x} h_i,
\end{equation}
where $h_i=x_{i+1/2}-x_{i-1/2}$ is the length of the cell $I_i$. 

Next for $k\geq 0$, we define the discontinuous finite element space $V_h$,
\begin{equation}
    V_h=\left\{g\in L^2(\Omega_x): \left.g\right|_{I_i}\in\P^k(I_i), \forall I_i,\,1\leq i\leq N_x\right\}.
\end{equation}
where $\P^k(I_i)$ is the space of polynomials of degree at most $k$ in each variable.

In order to approximate the the gPC coefficients via the DG method, it is necessary to solve the decomposed system given in Equation \eqref{eqn:gpc_matrix_system}.
The idea is to find $\bq_h=(\hat{q}_{1,h},\cdots,\hat{q}_{N,h})\in (V_h)^N$ that approximates $\bq=(\hat{q}_1,\cdots,\hat{q}_N)$ in \eqref{eqn:q_system}, since the matrix $\Lambda$ is diagonal we can decouple the system in $N$ transport equations,
\begin{equation}
    \frac{\partial \hat{q}_n}{\partial t}=\lambda_n\frac{\partial \hat{q}_n}{\partial x},\quad 0\leq n\leq N,
\end{equation}
Then for each $n$, we want to find $q_{n,h}$ such that for all $\varphi_h\in V_h$, such that
\begin{equation}\label{eqn:q_dg_element}
    \int_{I_i} (\hat{q}_{n,h})_t\varphi_h\,dx=\left.\widetilde{\lambda_n \hat{q}_{n,h}}\varphi_h\right|_{x_{i+1/2}}-\left.\widetilde{\lambda_n \hat{q}_{n,h}}\varphi_h\right|_{x_{i-1/2}}-\lambda_n\int_{I_i}\hat{q}_{n,h}\varphi_h\,dx
\end{equation}
The ``tilde'' functions denote the upwind flux:

\begin{equation}
    \lambda_n \widetilde{\hat{q}_{n,h}} =
    \begin{cases}
        \lambda_j\hat{q}^{-}_{j,h} & \text{if} \ \lambda_j \leq 0, \\
        \lambda_j\hat{q}^{+}_{j,h} & \text{if} \ \lambda_j > 0.
    \end{cases}
\end{equation}

we can also rewrite the method in the vector form, 
\begin{equation}
\label{eqn:q_dg_matrix}
    \int_{I_i} (\bq_h)_t\Phi_h\,dx=\left.\widetilde{\Lambda \bq_{h}}\Phi_h\right|_{x_{i+1/2}}-\left.\widetilde{\Lambda\bq_{h}}\Phi_h\right|_{x_{i-1/2}}-\Lambda\int_{I_i}\bq_h\Phi_h\,dx
\end{equation}
Then the idea is to solve \eqref{eqn:q_dg_element} (or \eqref{eqn:q_dg_matrix}) then we construct the approximation for $\v$, 
\begin{equation}
\v_h=\Seig\bq_h=(\hat{v}_{1,h},\cdots,\hat{v}_{N,h})
\end{equation}
Then the \emph{discontinuous Galerkin generalized polynomial chaos expansion (DG-gPC)} is given by 
\begin{equation}
    v_h(x,t,y)=\sum_{j=0}^N \hat{v}_{j,h}(x,t) P_j(y).
\end{equation}

\subsection{Temporal Discretizations}

For time integration, we apply a total variation diminishing (TVD) high-order Runge-Kutta method to solve the ODE resulting from the semidiscrete DG scheme \eqref{eqn:q_dg_element}-\eqref{eqn:q_dg_matrix},
\begin{equation}
    \frac{d}{dt}S_h=R(S_h).
\end{equation}
Such time-stepping methods are convex combinations of the Forward Euler time discretization. The commonly used third-order TVD Runge-Kutta method is given by
\begin{gather}
\begin{aligned}  
    S_h^{(1)}&=S_h^n+\Delta t R(S_h^n),\\  
    S_h^{(2)}&=\frac{3}{4}S_h^n+\frac{1}{4}S_h^{(1)}+\frac{1}{4}\Delta t R(S_h^{(1)}),\\  
    S_h^{n+1}&=\frac{1}{3}S_h^n+\frac{2}{3}S_h^{(2)}+\frac{2}{3}\Delta t R(S_h^{(2)}),  
\end{aligned}\label{eq:tvd-rk}  
\end{gather}
where $S_h^{n}$ represents a numerical approximation of the solution at the discrete time $t_n$. A detailed description of the TVD Runge-Kutta method can be found in \cite{gottlieb1998total}.

\subsection{Convergence of the approximated coefficients.}


When the truncated coefficients are approximated using a DG method, we then have the following theorem from \cite{JohnKat}:
\begin{thm}[$L^2$-error estimates]\label{thm:l2_error_estimates}
In addition to the hypothesis of \Cref{thm:spectral_convergence} assume that $\uhat{v}{n}(\cdot,t)\in \mathcal{H}^{k+2}(\Ox),\, n =0,\dots,N$ are the gPC coefficients defined by Equation \eqref{eqn:gpc_matrix_system} and $\uhat{v}{{n,h}}(\cdot,t)$ are the coefficients obtained using a DG approximation (\Cref{eqn:q_dg_matrix}). Then the error from approximating the gPC coefficients using the DG method is given by 

\begin{equation}
\norma{\uhat{v}{n}-\uhat{v}{{n,h}}}{\Ltwo}{{\Ox}}\leq C |\uhat{v}{n}(\cdot,0)|_{\mathcal{H}^{k+2}(\Ox)}\, h^{k+1},
\end{equation}


where $C$ depends on $k$, the numerical flux, and $T.$ Furthermore, 
\begin{equation}
    \expec\left[\norma{u-v_h}{L^2}{\Ox}^2\right]\leq C |\v|^2_{\mathcal{H}^{k+2}(\Ox)}h^{2(k+1)}+\frac{K}{N^{2m-1}}t,
\end{equation}  
where $m$ is the smoothness in random space.
\end{thm}

\begin{proof}
The proof of convergence in $\Ltwo$ for DG was shown in \cite{JohnKat} and hence gives rise to the first estimate.

For the second part we use the triangle inequality,

\begin{equation}
\expec\left[\norma{u-v_h}{L^2}{\Ox}^2\right]\leq \expec\left[\norma{u-v}{L^2}{\Ox}^2\right]+\expec\left[\norma{v-v_h}{L^2}{\Ox}^2\right].  
\end{equation}

The result then follows from using Theorem \ref{thm:spectral_convergence} and \cite{JohnKat}.
\end{proof}

The next important theorem builds on the previous Theorem and \cite{CLSS}.

\begin{thm}[Negative-order error estimates]\label{thm:coeff_neg_estimates} 
Let $\uhat{u}{n}(\cdot,0) \in L^2_{per}(\mathbb{R}^d;\mathbb{R}^m)\, \cap\, \mathcal{H}^r(\mathcal{D}{{\Omega_{x,1}}}),$ with ${\Omega_{x,0}} \subseteq {\Omega_{x,1}} \subset \Ox,$ $r\geq 0$ with $\uhat{u}{n}(\cdot,t), n =0,\dots,N$ being the coefficients defined by Equation \eqref{eqn:coeff_infinite_system}, $\uhat{v}{n}(\cdot,t)$ the coefficients in the gPC expansion, and $\uhat{v}{{n,h}}(\cdot,t)$ the gPC coefficients approximated using a discontinuous Galerkin approximation using piecewise polynomials of degree $k \leq r,\, \ell$. Then, for $\ell \geq 1,$
\begin{equation}
\label{eq:negnormest}
\expec\left[\norma{u-v_{h}}{\mathcal{H}^{-\ell}}{{\Omega_{x,0}}}\right]\leq C_0 \|\hat{v}_0(\cdot,0)\|_{r,\mathcal{D}\Ox}h^{s}+\frac{K}{N^{m-1/2}}t,
\end{equation}
where $s = \min\{2k+1,\ell+r/2+1\}$ and $m$ represents the smoothness in random space. 
\end{thm}

We note that generally, $r=2k.$

\begin{proof}
    Here, we simply emphasize the effect of looking at the expectation of the DG-gPC approximation.
\begin{align*}
    \expec\left[\norma{u-v_h}{\mathcal{H}^{-\ell}}{\Omega_{x,0}}\right] 
    &=\expec\left[\sup_{\Phi\in\mathcal{C}_0^\infty(\Omega_{x,0})}\, \frac{(u-v_h,\Phi)_{\Ltwo(\Ox)}(T)}{\norma{\Phi}{\mathcal{H}^{\ell}}{{\Omega_{x,0}}}}\right]\\
   &=  \sum_{n=0}^N\, \norma{\uhat{u}{n}-\uhat{v}{n,h}}{\mathcal{H}^{-\ell}}{\Omega_{x,0}} \expec\left[P_n(y)\right]+\sum_{n=N+1}^\infty\,  \norma{\uhat{u}{n}}{\mathcal{H}^{-\ell}}{\Omega_{x,0}} \expec\left[P_n(y)\right]\\
&= \norma{\uhat{u}{0}-\uhat{v}{0,h}}{\mathcal{H}^{-\ell}}{\Omega_{x,0}}\\
&= \norma{\uhat{u}{0}-\uhat{v}{0}}{\mathcal{H}^{-\ell}}{\Omega_{x,0}} + \norma{\uhat{v}{0}-\uhat{v}{0,h}}{\mathcal{H}^{-\ell}}{\Omega_{x,0}}\\
    &\leq \norma{\uhat{u}{0}-\uhat{v}{0}}{\mathcal{H}^{-\ell}}{\Omega_{x,0}} + C_0 \|\hat{v}_0(\cdot,0)\|_{r,\mathcal{D}\Ox}h^{r+1}.
\end{align*} 
since $\expec\left[P_n(y)\right]=\delta_{n,0}$. We note that bound on the negative-order norm for the DG coefficients was given in \cite{CLSS}. We only need to bound the first term of the last line of the inequality above. An application of  Theorem \ref{thm:spectral_convergence},  gives
\begin{align*}
\frac{\left( \uhat{u}{0}-\uhat{v}{0},\Phi(x)\right)_{\Ltwo(\Omega_{x,0})}(T)}{\norma{\Phi}{\mathcal{H}^{\ell}}{\Omega_{x,0}}} 
&\leq \frac{\norma{\uhat{u}{0}-\uhat{v}{0}}{L^2}{\Omega_{x,0}}}{\norma{\Phi}{\mathcal{H}^{\ell}}{\Omega_{x,0}}} \norma{\Phi}{L^2}{\Omega_{x,0}} \leq \frac{\norma{\uhat{u}{0}-\uhat{v}{0}}{L^2}{\Omega_{x,0}}}{\norma{\Phi}{\mathcal{H}^{\ell}}{\Omega_{x,0}}} \norma{\Phi}{\mathcal{H}^{\ell}}{\Omega_{x,0}} \\
&= \norma{\uhat{u}{0}-\uhat{v}{0}}{L^2}{\Omega_{x,0}} \leq \frac{K}{N^{m-1/2}}t.
\end{align*}

Hence, utilizing the above together we arrive at:
\begin{equation}
\expec\left[\norma{u-v_h}{\mathcal{H}^{-\ell}}{\Omega_{x,0}}\right]\leq C_0 \|\hat{v}_0(\cdot,0)\|_{r,\mathcal{D}\Ox}h^{s}+\frac{K}{N^{m-1/2}}t,
\end{equation}
where $s = \min\{2k+1,\ell+r/2+1\}$ and $m$ represents the smoothness in random space. 
\end{proof}

We note that the negative-order norm gives a measure of oscillations about zero \cite{CLSS}.  Hence it can assist in obtaining a better estimate to the statistical quantities of interest.

\section{SIAC filtering}
\label{secn:siac}

In this section, we describe an accuracy-enhancing filter that aids in reducing small scale noise that arises from the approximation of the coefficients via piecewise polynomial approximations.  The Smoothness-Increasing Accuracy-Enhancing (SIAC) filter was introduced as a post-processor by Bramble and Schatz \cite{bramble1977} and utilizes the ability to reproduce moments as discussed in Mock and Lax \cite{mock1978}.  The post-processor was then extended to discontinuous Galerkin approximations for hyperbolic conservation laws by Cockburn, Luskin, Shu, and S\"{u}li \cite{CLSS}.  The SIAC terminology was adopted to account for a broader class of filters \cite{SIACReview,Mirzargar2016}, such as developments to incorporate non-symmetric filters for post-processing near boundaries \cite{SRV,SIACLinf,XLiOne}, as well as multi-dimensional filtering, including a line filter and hexagonal filter \cite{LSIAC,HexSIAC}.  In the below, we provide a summary of the filter.

\subsection{The SIAC Filter}

In this article, the post-processed solution is obtained by convolving the approximated coefficients at the final time with the SIAC kernel function.  The symmetric form of the SIAC kernel based on B-splines is a linear combination of $r+1$ B-Splines of order $\ell$:
\[
\KH{}{r+1}{\ell}{x} = \FullK{r}{\BS{\ell}}{x},
\]
with $x_\gamma= -\lceil\frac{r}{2}\rceil+\gamma$. This allows for accelerating the convergence rate for a DG approximation using polynomials of degree $k$ from $k+1$ to $\min\{2k+1,\ell+r/2+1\}.$ Thus taking advantage of the underlying superconvergence properties at the Radau points \cite{Adjerid1} as well as in the dissipation and dispersion errors \cite{Guo}. An alternative expression for the kernel function is in terms of its Fourier Transform is given by 
\[
\FKH{}{r+1}{\ell}{\xi} =\left(\sinc\left(\frac{\xi}{2}\right)\right)^n\left(c_{r/2}+2\sum_{\gamma=0}^{r/2-1}c_\gamma \cos(\gamma \xi)\right).
\]

We note that it is only necessary to post-process the coefficients described by Equation \eqref{eqn:model} at the final time.   As the weights for the post-processing kernel are defined through moment conditions, we are then able to extract the information contained in the negative-order norm \eqref{eq:negnormest}  \cite{mock1978,bramble1977,CLSS}. 

The SIAC filter is incorporated into the procedure in a simple way: 
\begin{itemize}
\item First, approximate the coefficients $\mathbf{\hat{v}_h}=(\hat{v}_{0,h}, \hat{v}_{1,h}, \dots, \hat{v}_{N,0})^T$ using the discontinuous Galerkin method. 
\item Next, at the final time, filter the approximated coefficients: 
\begin{equation}
    \mathbf{\hat{v}}_h^*=K_h^{(r+1,\ell)}\star \mathbf{\hat{v}}_h.
\end{equation}
\item Then the modified gPC construction using the filtered DG coefficients is given by
\begin{equation}
v_h^*(x,t,y)=\sum_{\ell=0}^N \hat{v}^*_{\ell,h}(x,t)P_{\ell}(y),
\end{equation}
it is actually the moments that are the quantities of interest. Hence, we concentrate on those moments, noting that the first two moments, mean and variance, are the quantities of interest.
\end{itemize}
\subsection{Theory: Superconvergence-Extracting Post-Processing} 
\label{secn:theory}

As noted in the previous section, it is the information in the negative-order norm, Equation \eqref{eq:negnormest}, that is utilized to extract the extra information using the SIAC filter \cite{CLSS}.  The negative-order norm is used to bound the $L^2-$norm of the errors in the post-processed DG approximation for the coefficients:
\begin{thm}\label{eqn:coeff_bound+
}
    Let $\uhat{u}{n}(\cdot,t)\in L^2_{per}(\mathbb{R}^d;\mathbb{R}^m)\, \cap\, \mathcal{H}^r(\mathcal{D}\Omega_{x,1})$ be a solution to Equation \eqref{eqn:model} and $\uhat{v}{h,n}$ a DG approximation to $\hat{v}_{n}$, the coefficients utilizing gPC expansion. Given a convolution kernel made up of a linear combination of $r+1$ translated B-splines of order $\ell,\quad K_H^{r+1,\ell}=\frac{1}{H}\sum_{\gamma = 0}^r\, c_\gamma B_\ell\left(\frac{x}{H}-x_{\gamma}\right),$ then
    \begin{equation}
    \label{eq:SIACest}
        \expec\left[\norma{u-v_h^*}{L^2}{{\Omega_{x,0}}}^2\right]\leq C h^{2s^*}+\frac{K}{N^{2m-1}}t,
    \end{equation}
    with $s^*=\min\{2k+1,\ell+r/2+1\}.$
\end{thm}

\begin{proof}
Notice that  

\begin{equation}
    \expec\left[\norma{u-v_h^*}{L^2}{{\Omega_{x,0}}}^2\right]\leq \underbrace{\expec\left[\norma{u-v}{L^2}{{\Omega_{x,0}}}^2\right]}_{\text{gPC approximation}}+\underbrace{\expec\left[\norma{v-v^*}{L^2}{{\Omega_{x,0}}}^2\right]}_{\text{Filter construction}} +\underbrace{\expec\left[\norma{v^*-v_h^*}{L^2}{{\Omega_{x,0}}}^2\right]}_{\text{Filtered DG}}
\end{equation}
From Theorem \ref{thm:spectral_convergence}, we can bound the first term:
\[
\expec\left[\norma{u-v}{L^2}{{\Omega_{x,0}}}^2\right]\leq \frac{K}{N^{2m-1}}t.
\]
The bound on the second term relies solely on the number of moment conditions,
\begin{align*}
\expec[\norma{v - v^*}{L^2}{{\Omega_{x,0}}}^2] \leq  C_K h^{4r+2},
\end{align*}
with $r$ being the number of moments.  Note that this order is actually increased by two orders for the symmetric kernel.

To bound the third term, we utilize the results in \cite{CLSS,SIACLinf} to obtain 
\begin{align*}
\expec[\norma{v - v_h^*}{L^2}{{\Omega_{x,0}}}^2] &= \sum_{n=0}^N\norma{v_n - v_{n,h}^*}{L^2}{{\Omega_{x,0}}}^2 \leq  C_K h^{2s^*},
\end{align*}
    with $s^*=\min\{2k+1,\ell+r/2+1\}.$
\end{proof}

Note that if the number of elements for the DG approximation are chosen proportional to the number of gPc expansion coefficients, $N_{DG} = \alpha N$, it is easy to understand the affect of filtering and how to balance the error estimate for the various terms.   Note that in this case $h = \frac{|\Ox|}{\alpha N}.$ If we consider the usual $L^2-$error estimate for the DG approximation of the gPC coefficients, balancing the error terms leads to :
\[
C |v(\cdot,0)|_{\mathcal{H}^{k+2}(\Ox)}\left(\frac{|\Ox|}{\alpha N}\right)^{2(k+1)} + \frac{K}{N^{2m-1}}t \leq \max\left(C |v(\cdot,0)|_{\mathcal{H}^{k+2}(\Ox)}\left(\frac{|\Ox|}{\alpha }\right)^{2(k+1)},Kt\right)\, \left(\frac{1}{ N}\right)^{2s},
\]
with $s=\min\left(k+1,m-1/2\right)$.
If we instead use the filtered approximation, this becomes
\[
C_F h^{2s^*} + \frac{K}{N^{2m-1}}t \leq \max\left(C_F\left(\frac{| \Ox |}{\alpha }\right)^{2s^*},Kt\right)\, \left(\frac{1}{ N}\right)^{2s^*},
\]
with $s^*=\min\left(2k+1,\ell+r/2+1,m-1/2\right)$. Hence, without filtering, $k = \frac{2m-3}{2}$.  With filtering, assuming the full $2k+1$ convergence can be achieved, this becomes $k = \frac{2m-3}{4}$.   
If we consider the above relations and the points-per-wavelength,  $N_x(k+1)\geq \pi,$ \cite{Hes08W}, with the choice of polynomial degree based on $m,$ then, without the filter, $N_x \geq \frac{2\pi}{2m-1},$ whereas with the filter, $N_x \geq \frac{4\pi}{2m-3}.$





\section{Numerical Experiments}\label{secn:numerical_experiments}
In this section, we present numerical examples to support the theoretical results and demonstrate how the SIAC filter enhances the first few statistical moments, which are the quantities of interest. Since the computed coefficients are not the primary focus, the Polynomial Chaos expansion must be post-processed to extract meaningful information. Depending on the situation, either the probability density function, statistical moments, or quantiles of $u$ are the quantities of interest. Applying  the SIAC filter helps refine these key aspects of the model response.

Using the orthonormality of the polynomial chaos basis one can obtain the mean, $\mu_{v_h}$:
\begin{align}
\mu_{v_h} &= \expec[v_h(x,t;y)]=\int_{\Oy} \, v_h(x,t;y)\rho(y)\, dy\\
          &=\int_{\Oy} \, \left(\sum_{n=0}^N\, \uhat{v}{n,h}(x,t)(x,t)P_n(y)\right)\rho(y)\, dy=\uhat{v}{0,h}(x,t)
\end{align}
and the Variance, $\sigma_u^2$: 
\begin{align}
    \sigma_{v_h}^2=&\expec[(v_h-\mu_{v_h})^2]= \expec[v_h^2]-\expec[v_h]^2 \notag\\
    = &\int_{\Oy} \, \left(\sum_{n=0}^N\, \uhat{v}{n,h}(x,t)(x,t)P_n(y)\right)^2\rho(y)\, dy - \left(\uhat{v}{0,h}(x,t)\right)^2.\notag \\
    =& \sum_{n=1}^N\, \left(\uhat{v}{n,h}(x,t)\right)^2
\end{align}

For our DG-gPC solutions we examine the following five error measures: 
\begin{itemize}
\item[] Mean-square error
\begin{equation}
    \expec[\|u-v_h\|_{L^2(\Omega_x)}^2],
    \label{eq:mean_square_error}
\end{equation}
\item[] $L^p$-error in the mean:
\begin{equation}
    \norma{\mu-\mu_h}{L^p}{\Omega_x},\quad p=\infty,\,2,
    \label{eq:lp_mean_error}
\end{equation}
\item[] $L^p$-error in the variance:
\begin{equation}
    \norma{\sigma-\sigma_h}{L^p}{\Omega_x},\quad p=\infty,\,2,
    \label{eq:lp_std_error}
\end{equation}
\end{itemize}

We illustrate the effectiveness of the filter and the affects of the different parameters in the filter by considering the problem \eqref{eqn:model} with a periodic boundary condition in physical space. The periodicity ensures that no errors are induced from specifying boundary conditions. The model problem that we consider is
\begin{equation}
    \begin{split}
        &\frac{\partial u}{\partial t} = y \frac{\partial u}{\partial x}, \qquad  0 < x < 2\pi,\, t \geq 0 \\
        &u(x,0,y) = \cos(x), \quad  0 < x < 2\pi.
    \end{split}
\end{equation}
The exact solution is given by
\begin{equation}
    u_{\mathrm{exact}}(x,t,y)=\cos(x+yt). 
\end{equation}

The exact formulas for the quantities of interest, i.e. mean $\mu_u$ and variance $\sigma_u,$ are given by
\begin{equation}
        \mu_{\mathrm{exact}}=\frac{\cos(x)\sin(t)}{t}
        \label{eq:exact_mean_periodic}
\end{equation}
\begin{equation}
        \sigma_{\mathrm{exact}}=\frac{t}{2}+\frac{\cos(2x)\sin(2t)}{4t}-\frac{\cos^2(x)\sin^2(t)}{t^2},
        \label{eq:exact_variance_periodic}
    \end{equation}
    respectively.
    
 \Cref{tbl:mean_square_error,tbl:linf_mean_error,tbl:ltwo_mean_error,tbl:linf_var_error,tbl:ltwo_variance_error}, and  results are demonstrated for the system \eqref{eqn:gpc_matrix_system} at time $T=1$ for a different number of polynomial chaos basis functions, $N$.  The SIAC filter is then applied to the DG approximation of the chaos-coefficients and then both the filtered mean $\mu_h^*$ and variance $\sigma_h^*$ are calculated for the filtered solution, $v^*_h$.  Results are then compared with $u_{\mathrm{exact}},\,\mu_{\mathrm{exact}}$ and $\sigma_{\mathrm{exact}}$. As a time integrator we used a TVD-RK method \eqref{eq:tvd-rk}. To have a clean presentation of the numerical experiments, we make the following change of notation for the mesh size: $h=\Delta x$. To ensure the spatial order dominates, we take $\Delta t= \mathrm{CFL}\cdot\Delta x/\lambda_{\max}$ for $\mathbb{P}^1$, and $\Delta t=\mathrm{CFL}\Delta x^{5/3}/\lambda_{\max}$ for $\mathbb{P}^2$. Here $\lambda_{\max}=\max\{|\lambda|:\lambda\text{ is an eigenvalue of}\,\mathbf{A}\}$, where $\mathbf{A}$ is the system matrix. From the tables we can see that in all measures, for $N$ large enough, we observe the predicted $(k+1)$-th  order of convergence for the DG-gPc solution before post-processing for $u$, $\mu$ and $\sigma$ ($(k+1)^2$ for the mean square error in Table \ref{tbl:mean_square_error}). We can clearly see that we improve the error to at least $O(h^{2k+1})$ ($O(h^{4k+2})$ for the mean-square error in Table \ref{tbl:mean_square_error}).

\begin{table}[ht]
\centering
\begin{tabular}{lcccccccc}
\hline
\multicolumn{9}{l}{Before post-processing}                                                                                                                                                                   \\ \hline
                     & \multicolumn{1}{l}{} & \multicolumn{1}{l}{} & \multicolumn{1}{l}{} & \multicolumn{1}{l}{} & \multicolumn{1}{l}{} & \multicolumn{1}{l}{} & \multicolumn{1}{l}{} & \multicolumn{1}{l}{} \\
                     & $N=5$                &                      & $N=6$                &                      & $N=7$                &                      & $N=8$                &                      \\
Mesh                 & Error $u$            & Order                & Error $u$            & Order                & Error $u$            & Order                & Error $u$            & Order                \\ \hline
                     & \multicolumn{1}{l}{} & \multicolumn{1}{l}{} & \multicolumn{1}{l}{} & \multicolumn{1}{l}{} & \multicolumn{1}{l}{} & \multicolumn{1}{l}{} & \multicolumn{1}{l}{} & \multicolumn{1}{l}{} \\
$\mathbb{P}^1$       &                      &                      &                      &                      &                      &                      &                      &                      \\
10                   & 1.82E-03             &                      & 1.53E-03             &                      & 1.44E-03             &                      & 1.46E-03             &                      \\
20                   & 1.27E-04             & 3.84                 & 1.00E-04             & 3.93                 & 9.46E-05             & 3.93                 & 9.59E-05             & 3.93                 \\
40                   & 8.15E-06             & 3.97                 & 6.33E-06             & 3.99                 & 5.99E-06             & 3.98                 & 6.09E-06             & 3.98                 \\
80                   & 5.13E-07             & 3.99                 & 3.96E-07             & 4.00                 & 3.76E-07             & 4.00                 & 3.83E-07             & 3.99                 \\
160                  & 3.42E-08             & 3.91                 & 2.47E-08             & 4.00                 & 2.35E-08             & 4.00                 & 2.40E-08             & 4.00                 \\
                     & \multicolumn{1}{l}{} & \multicolumn{1}{l}{} & \multicolumn{1}{l}{} & \multicolumn{1}{l}{} & \multicolumn{1}{l}{} & \multicolumn{1}{l}{} & \multicolumn{1}{l}{} & \multicolumn{1}{l}{} \\
$\mathbb{P}^2$       &                      &                      &                      &                      &                      &                      &                      &                      \\
10                   & 6.10E-06             &                      & 4.62E-06             &                      & 4.13E-06             &                      & 4.10E-06             &                      \\
20                   & 8.31E-08             & 6.20                 & 6.60E-08             & 6.13                 & 6.28E-08             & 6.04                 & 6.19E-08             & 6.05                 \\
40                   & 3.47E-09             & 4.58                 & 1.03E-09             & 6.01                 & 9.67E-10             & 6.02                 & 9.81E-10             & 5.98                 \\
80                   & 2.19E-09             & 0.66                 & 2.77E-11             & 6.40                 & 1.51E-11             & 6.00                 & 1.54E-11             & 6.00                 \\
160                  & 2.17E-09             & 0.01                 & 1.21E-11             & 1.19                 & 2.80E-13             & 5.75                 & 2.40E-13             & 6.00                 \\ \hline
\multicolumn{9}{l}{After post-processing}                                                                                                                                                                    \\ \hline
                     & \multicolumn{1}{l}{} & \multicolumn{1}{l}{} & \multicolumn{1}{l}{} & \multicolumn{1}{l}{} & \multicolumn{1}{l}{} & \multicolumn{1}{l}{} & \multicolumn{1}{l}{} & \multicolumn{1}{l}{} \\
\multicolumn{1}{c}{} & $N=5$                &                      & $N=6$                &                      & $N=7$                &                      & $N=8$                &                      \\
Mesh                 & Error $u^*$          & Order                & Error $u^*$          & Order                & Error $u^*$          & Order                & Error $u^*$          & Order                \\ \hline
                     & \multicolumn{1}{l}{} & \multicolumn{1}{l}{} & \multicolumn{1}{l}{} & \multicolumn{1}{l}{} & \multicolumn{1}{l}{} & \multicolumn{1}{l}{} & \multicolumn{1}{l}{} & \multicolumn{1}{l}{} \\
$\mathbb{P}^1$       &                      &                      &                      &                      &                      &                      &                      &                      \\
10                   & 3.65E-05             &                      & 3.58E-05             &                      & 3.64E-05             &                      & 3.65E-05             &                      \\
20                   & 3.72E-07             & 6.62                 & 3.47E-07             & 6.69                 & 3.58E-07             & 6.67                 & 3.62E-07             & 6.65                 \\
40                   & 7.88E-09             & 5.56                 & 4.25E-09             & 6.35                 & 4.22E-09             & 6.41                 & 4.32E-09             & 6.39                 \\
80                   & 2.41E-09             & 1.71                 & 9.49E-11             & 5.48                 & 5.57E-11             & 6.24                 & 5.83E-11             & 6.21                 \\
160                  & 2.19E-09             & 0.14                 & 1.63E-11             & 2.54                 & 7.52E-13             & 6.21                 & 8.33E-13             & 6.13                 \\
                     & \multicolumn{1}{l}{} & \multicolumn{1}{l}{} & \multicolumn{1}{l}{} & \multicolumn{1}{l}{} & \multicolumn{1}{l}{} & \multicolumn{1}{l}{} & \multicolumn{1}{l}{} & \multicolumn{1}{l}{} \\
$\mathbb{P}^2$       &                      &                      &                      &                      &                      &                      &                      &                      \\
10                   & 1.22E-07             &                      & 1.24E-07             &                      & 1.23E-07             &                      & 1.23E-07             &                      \\
20                   & 2.17E-09             & 5.82                 & 6.78E-11             & 10.84                & 3.62E-11             & 11.73                & 3.58E-11             & 11.74                \\
40                   & 2.17E-09             & 0.00                 & 1.22E-11             & 2.47                 & 5.69E-14             & 9.32                 & 9.23E-15             & 11.92                \\
80                   & 2.17E-09             & 0.00                 & 1.19E-11             & 0.04                 & 4.46E-14             & 0.35                 & 1.27E-16             & 6.18                 \\
160                  & 2.17E-09             & 0.00                 & 1.19E-11             & 0.00                 & 4.46E-14             & 0.00                 & 1.47E-16             & -0.21                \\ \hline
\end{tabular}
\caption{Mean square errors for the numerical solution (top) and the post-processed solution (bottom) for the periodic problem using the discontinuous Galerkin method using a $\mathbb{P}^k$ polynomial approximation in physical space $x$.}
\label{tbl:mean_square_error}
\end{table}

\begin{table}[ht]
\centering
\begin{tabular}{lcccccccc}
\hline
\multicolumn{9}{l}{Before post-processing}                                                                                                                                                                   \\ \hline
                     & \multicolumn{1}{l}{} & \multicolumn{1}{l}{} & \multicolumn{1}{l}{} & \multicolumn{1}{l}{} & \multicolumn{1}{l}{} & \multicolumn{1}{l}{} & \multicolumn{1}{l}{} & \multicolumn{1}{l}{} \\
                     & $N=5$              &                      & $N=6$                &                      & $N=7$                &                      & $N=8$                &                      \\
Mesh                 & Error $\mu$          & Order                & Error $\mu$          & Order                & Error $\mu$          & Order                & Error $\mu$          & Order                \\ \hline
                     & \multicolumn{1}{l}{} & \multicolumn{1}{l}{} & \multicolumn{1}{l}{} & \multicolumn{1}{l}{} & \multicolumn{1}{l}{} & \multicolumn{1}{l}{} & \multicolumn{1}{l}{} & \multicolumn{1}{l}{} \\
$\mathbb{P}^1$       &                      &                      &                      &                      &                      &                      &                      &                      \\
10                   & 2.56E-02             &                      & 2.55E-02             &                      & 2.56E-02             &                      & 2.55E-02             &                      \\
20                   & 6.89E-03             & 1.89                 & 6.86E-03             & 1.89                 & 6.89E-03             & 1.89                 & 6.87E-03             & 1.89                 \\
40                   & 1.79E-03             & 1.95                 & 1.77E-03             & 1.95                 & 1.78E-03             & 1.95                 & 1.77E-03             & 1.95                 \\
80                   & 4.55E-04             & 1.97                 & 4.50E-04             & 1.98                 & 4.54E-04             & 1.97                 & 4.51E-04             & 1.98                 \\
160                  & 1.15E-04             & 1.99                 & 1.13E-04             & 1.99                 & 1.15E-04             & 1.99                 & 1.14E-04             & 1.99                 \\
                     & \multicolumn{1}{l}{} & \multicolumn{1}{l}{} & \multicolumn{1}{l}{} & \multicolumn{1}{l}{} & \multicolumn{1}{l}{} & \multicolumn{1}{l}{} & \multicolumn{1}{l}{} & \multicolumn{1}{l}{} \\
$\mathbb{P}^2$       &                      &                      &                      &                      &                      &                      &                      &                      \\
10                   & 1.61E-03             &                      & 1.74E-03             &                      & 1.52E-03             &                      & 1.72E-03             &                      \\
20                   & 2.20E-04             & 2.87                 & 2.13E-04             & 3.03                 & 2.24E-04             & 2.76                 & 2.11E-04             & 3.03                 \\
40                   & 2.70E-05             & 3.03                 & 2.66E-05             & 3.00                 & 2.67E-05             & 3.07                 & 2.67E-05             & 2.99                 \\
80                   & 3.37E-06             & 3.00                 & 3.33E-06             & 3.00                 & 3.36E-06             & 2.99                 & 3.34E-06             & 3.00                 \\
160                  & 4.21E-07             & 3.00                 & 4.16E-07             & 3.00                 & 4.20E-07             & 3.00                 & 4.17E-07             & 3.00                 \\ \hline
\multicolumn{9}{l}{After post-processing}                                                                                                                                                                    \\ \hline
                     & \multicolumn{1}{l}{} & \multicolumn{1}{l}{} & \multicolumn{1}{l}{} & \multicolumn{1}{l}{} & \multicolumn{1}{l}{} & \multicolumn{1}{l}{} & \multicolumn{1}{l}{} & \multicolumn{1}{l}{} \\
\multicolumn{1}{c}{} & $N=5$                &                      & $N=6$                &                      & $N=7$                &                      & $N=8$                &                      \\
Mesh                 & Error $\mu^*$        & Order                & Error $\mu^*$        & Order                & Error $\mu^*$        & Order                & Error $\mu^*$        & Order                \\ \hline
                     & \multicolumn{1}{l}{} & \multicolumn{1}{l}{} & \multicolumn{1}{l}{} & \multicolumn{1}{l}{} & \multicolumn{1}{l}{} & \multicolumn{1}{l}{} & \multicolumn{1}{l}{} & \multicolumn{1}{l}{} \\
$\mathbb{P}^1$       &                      &                      &                      &                      &                      &                      &                      &                      \\
10                   & 2.80E-03             &                      & 2.78E-03             &                      & 2.79E-03             &                      & 2.78E-03             &                      \\
20                   & 2.67E-04             & 3.39                 & 2.61E-04             & 3.41                 & 2.65E-04             & 3.40                 & 2.62E-04             & 3.41                 \\
40                   & 2.74E-05             & 3.28                 & 2.64E-05             & 3.31                 & 2.72E-05             & 3.29                 & 2.66E-05             & 3.30                 \\
80                   & 3.04E-06             & 3.17                 & 2.89E-06             & 3.19                 & 3.01E-06             & 3.18                 & 2.92E-06             & 3.19                 \\
160                  & 3.55E-07             & 3.10                 & 3.35E-07             & 3.11                 & 3.51E-07             & 3.10                 & 3.40E-07             & 3.11                 \\
                     & \multicolumn{1}{l}{} & \multicolumn{1}{l}{} & \multicolumn{1}{l}{} & \multicolumn{1}{l}{} & \multicolumn{1}{l}{} & \multicolumn{1}{l}{} & \multicolumn{1}{l}{} & \multicolumn{1}{l}{} \\
$\mathbb{P}^2$       &                      &                      &                      &                      &                      &                      &                      &                      \\
10                   & 1.67E-04             &                      & 1.67E-04             &                      & 1.67E-04             &                      & 1.67E-04             &                      \\
20                   & 2.84E-06             & 5.87                 & 2.84E-06             & 5.88                 & 2.84E-06             & 5.88                 & 2.84E-06             & 5.88                 \\
40                   & 4.80E-08             & 5.89                 & 4.77E-08             & 5.90                 & 4.79E-08             & 5.89                 & 4.77E-08             & 5.89                 \\
80                   & 8.44E-10             & 5.83                 & 8.33E-10             & 5.84                 & 8.40E-10             & 5.83                 & 8.34E-10             & 5.84                 \\
160                  & 1.69E-11             & 5.64                 & 1.58E-11             & 5.72                 & 1.60E-11             & 5.71                 & 1.58E-11             & 5.72                 \\ \hline
\end{tabular}
\caption{$L^{\infty}$ errors for the approximation to the mean (top) and the post-processed mean (bottom) for the periodic problem using the discontinuous Galerkin method using a $\mathbb{P}^k$ polynomial approximation in physical space $x$.}
\label{tbl:linf_mean_error}
\end{table}


\begin{table}[ht]
\centering
\begin{tabular}{lcccccccc}
\hline
\multicolumn{9}{l}{Before post-processing}                                                                                                                                                                   \\ \hline
                     & \multicolumn{1}{l}{} & \multicolumn{1}{l}{} & \multicolumn{1}{l}{} & \multicolumn{1}{l}{} & \multicolumn{1}{l}{} & \multicolumn{1}{l}{} & \multicolumn{1}{l}{} & \multicolumn{1}{l}{} \\
                     & $N=5$                &                      & $N=6$                &                      & $N=7$                &                      & $N=8$                &                      \\
Mesh                 & Error $\mu$          & Order                & Error $\mu$          & Order                & Error $\mu$          & Order                & Error $\mu$          & Order                \\ \hline
                     & \multicolumn{1}{l}{} & \multicolumn{1}{l}{} & \multicolumn{1}{l}{} & \multicolumn{1}{l}{} & \multicolumn{1}{l}{} & \multicolumn{1}{l}{} & \multicolumn{1}{l}{} & \multicolumn{1}{l}{} \\
$\mathbb{P}^1$       &                      &                      &                      &                      &                      &                      &                      &                      \\
10                   & 9.94E-03             &                      & 9.80E-03             &                      & 9.94E-03             &                      & 9.82E-03             &                      \\
20                   & 2.56E-03             & 1.96                 & 2.54E-03             & 1.95                 & 2.56E-03             & 1.96                 & 2.55E-03             & 1.95                 \\
40                   & 6.57E-04             & 1.96                 & 6.49E-04             & 1.97                 & 6.55E-04             & 1.97                 & 6.51E-04             & 1.97                 \\
80                   & 1.66E-04             & 1.98                 & 1.64E-04             & 1.99                 & 1.66E-04             & 1.98                 & 1.65E-04             & 1.98                 \\
160                  & 4.19E-05             & 1.99                 & 4.12E-05             & 1.99                 & 4.18E-05             & 1.99                 & 4.14E-05             & 1.99                 \\
                     & \multicolumn{1}{l}{} & \multicolumn{1}{l}{} & \multicolumn{1}{l}{} & \multicolumn{1}{l}{} & \multicolumn{1}{l}{} & \multicolumn{1}{l}{} & \multicolumn{1}{l}{} & \multicolumn{1}{l}{} \\
$\mathbb{P}^2$       &                      &                      &                      &                      &                      &                      &                      &                      \\
10                   & 5.43E-04             &                      & 5.26E-04             &                      & 5.29E-04             &                      & 5.30E-04             &                      \\
20                   & 6.68E-05             & 3.02                 & 6.75E-05             & 2.96                 & 6.75E-05             & 2.97                 & 6.73E-05             & 2.98                 \\
40                   & 8.57E-06             & 2.96                 & 8.48E-06             & 2.99                 & 8.58E-06             & 2.98                 & 8.51E-06             & 2.98                 \\
80                   & 1.08E-06             & 2.99                 & 1.07E-06             & 2.99                 & 1.08E-06             & 2.99                 & 1.07E-06             & 2.99                 \\
160                  & 1.36E-07             & 2.99                 & 1.34E-07             & 3.00                 & 1.35E-07             & 2.99                 & 1.34E-07             & 3.00                 \\ \hline
\multicolumn{9}{l}{After post-processing}                                                                                                                                                                    \\ \hline
                     & \multicolumn{1}{l}{} & \multicolumn{1}{l}{} & \multicolumn{1}{l}{} & \multicolumn{1}{l}{} & \multicolumn{1}{l}{} & \multicolumn{1}{l}{} & \multicolumn{1}{l}{} & \multicolumn{1}{l}{} \\
\multicolumn{1}{c}{} & $N=5$                &                      & $N=6$                &                      & $N=7$                &                      & $N=8$                &                      \\
Mesh                 & Error $\mu^*$        & Order                & Error $\mu^*$        & Order                & Error $\mu^*$        & Order                & Error $\mu^*$        & Order                \\ \hline
                     & \multicolumn{1}{l}{} & \multicolumn{1}{l}{} & \multicolumn{1}{l}{} & \multicolumn{1}{l}{} & \multicolumn{1}{l}{} & \multicolumn{1}{l}{} & \multicolumn{1}{l}{} & \multicolumn{1}{l}{} \\
$\mathbb{P}^1$       &                      &                      &                      &                      &                      &                      &                      &                      \\
10                   & 1.95E-03             &                      & 1.93E-03             &                      & 1.94E-03             &                      & 1.93E-03             &                      \\
20                   & 1.84E-04             & 3.40                 & 1.79E-04             & 3.43                 & 1.82E-04             & 3.41                 & 1.80E-04             & 3.42                 \\
40                   & 1.90E-05             & 3.27                 & 1.83E-05             & 3.29                 & 1.89E-05             & 3.27                 & 1.85E-05             & 3.29                 \\
80                   & 2.13E-06             & 3.16                 & 2.02E-06             & 3.18                 & 2.10E-06             & 3.16                 & 2.04E-06             & 3.18                 \\
160                  & 2.50E-07             & 3.09                 & 2.36E-07             & 3.10                 & 2.47E-07             & 3.09                 & 2.39E-07             & 3.10                 \\
                     & \multicolumn{1}{l}{} & \multicolumn{1}{l}{} & \multicolumn{1}{l}{} & \multicolumn{1}{l}{} & \multicolumn{1}{l}{} & \multicolumn{1}{l}{} & \multicolumn{1}{l}{} & \multicolumn{1}{l}{} \\
$\mathbb{P}^2$       &                      &                      &                      &                      &                      &                      &                      &                      \\
10                   & 1.17E-04             &                      & 1.17E-04             &                      & 1.17E-04             &                      & 1.17E-04             &                      \\
20                   & 2.00E-06             & 5.87                 & 2.00E-06             & 5.88                 & 2.00E-06             & 5.87                 & 1.91E-06             & 5.94                 \\
40                   & 3.38E-08             & 5.89                 & 3.36E-08             & 5.89                 & 3.37E-08             & 5.89                 & 3.36E-08             & 5.83                 \\
80                   & 5.95E-10             & 5.83                 & 5.87E-10             & 5.84                 & 5.92E-10             & 5.83                 & 5.88E-10             & 5.84                 \\
160                  & 1.19E-11             & 5.64                 & 1.11E-11             & 5.72                 & 1.13E-11             & 5.71                 & 1.12E-11             & 5.72                 \\ \hline
\end{tabular}
\caption{$L^{2}$ errors for the approximation to the mean (top) and the post-processed mean (bottom) for the periodic problem using the discontinuous Galerkin method using a $\mathbb{P}^k$ polynomial approximation in physical space $x$.}
\label{tbl:ltwo_mean_error}
\end{table}


\begin{table}[ht]
\centering
\begin{tabular}{lcccccccc}
\hline
\multicolumn{9}{l}{Before post-processing}                                                                                                                                                                   \\ \hline
                     & \multicolumn{1}{l}{} & \multicolumn{1}{l}{} & \multicolumn{1}{l}{} & \multicolumn{1}{l}{} & \multicolumn{1}{l}{} & \multicolumn{1}{l}{} & \multicolumn{1}{l}{} & \multicolumn{1}{l}{} \\
                     & $N=5$                &                      & $N=6$                &                      & $N=7$                &                      & $N=8$                &                      \\
Mesh                 & Error $\sigma$       & Order                & Error $\sigma$       & Order                & Error $\sigma$       & Order                & Error $\sigma$       & Order                \\ \hline
                     & \multicolumn{1}{l}{} & \multicolumn{1}{l}{} & \multicolumn{1}{l}{} & \multicolumn{1}{l}{} & \multicolumn{1}{l}{} & \multicolumn{1}{l}{} & \multicolumn{1}{l}{} & \multicolumn{1}{l}{} \\
$\mathbb{P}^1$       &                      &                      &                      &                      &                      &                      &                      &                      \\
10                   & 1.71E-02             &                      & 1.68E-02             &                      & 1.69E-02             &                      & 1.69E-02             &                      \\
20                   & 5.18E-03             & 1.72                 & 5.01E-03             & 1.75                 & 5.15E-03             & 1.72                 & 5.05E-03             & 1.74                 \\
40                   & 1.34E-03             & 1.95                 & 1.30E-03             & 1.95                 & 1.33E-03             & 1.95                 & 1.31E-03             & 1.95                 \\
80                   & 3.40E-04             & 1.98                 & 3.28E-04             & 1.98                 & 3.38E-04             & 1.98                 & 3.31E-04             & 1.98                 \\
160                  & 8.55E-05             & 1.99                 & 8.26E-05             & 1.99                 & 8.50E-05             & 1.99                 & 8.33E-05             & 1.99                 \\
                     & \multicolumn{1}{l}{} & \multicolumn{1}{l}{} & \multicolumn{1}{l}{} & \multicolumn{1}{l}{} & \multicolumn{1}{l}{} & \multicolumn{1}{l}{} & \multicolumn{1}{l}{} & \multicolumn{1}{l}{} \\
$\mathbb{P}^2$       &                      &                      &                      &                      &                      &                      &                      &                      \\
10                   & 1.65E-03             &                      & 1.60E-03             &                      & 1.60E-03             &                      & 1.61E-03             &                      \\
20                   & 1.98E-04             & 3.06                 & 1.92E-04             & 3.06                 & 2.02E-04             & 2.99                 & 1.90E-04             & 3.09                 \\
40                   & 2.46E-05             & 3.01                 & 2.33E-05             & 3.04                 & 2.44E-05             & 3.05                 & 2.36E-05             & 3.01                 \\
80                   & 3.03E-06             & 3.02                 & 2.87E-06             & 3.02                 & 3.00E-06             & 3.02                 & 2.90E-06             & 3.02                 \\
160                  & 3.78E-07             & 3.01                 & 3.56E-07             & 3.01                 & 3.73E-07             & 3.01                 & 3.60E-07             & 3.01                 \\ \hline
\multicolumn{9}{l}{After post-processing}                                                                                                                                                                    \\ \hline
                     & \multicolumn{1}{l}{} & \multicolumn{1}{l}{} & \multicolumn{1}{l}{} & \multicolumn{1}{l}{} & \multicolumn{1}{l}{} & \multicolumn{1}{l}{} & \multicolumn{1}{l}{} & \multicolumn{1}{l}{} \\
\multicolumn{1}{c}{} & $N=5$                &                      & $N=6$                &                      & $N=7$                &                      & $N=8$                &                      \\
Mesh                 & Error $\sigma^*$     & Order                & Error $\sigma^*$     & Order                & Error $\sigma^*$     & Order                & Error $\sigma^*$     & Order                \\ \hline
                     & \multicolumn{1}{l}{} & \multicolumn{1}{l}{} & \multicolumn{1}{l}{} & \multicolumn{1}{l}{} & \multicolumn{1}{l}{} & \multicolumn{1}{l}{} & \multicolumn{1}{l}{} & \multicolumn{1}{l}{} \\
$\mathbb{P}^1$       &                      &                      &                      &                      &                      &                      &                      &                      \\
10                   & 2.04E-03             &                      & 2.05E-03             &                      & 2.05E-03             &                      & 2.05E-03             &                      \\
20                   & 2.12E-04             & 3.27                 & 2.12E-04             & 3.27                 & 2.12E-04             & 3.27                 & 2.12E-04             & 3.27                 \\
40                   & 2.41E-05             & 3.13                 & 2.42E-05             & 3.13                 & 2.42E-05             & 3.13                 & 2.42E-05             & 3.13                 \\
80                   & 2.85E-06             & 3.08                 & 2.87E-06             & 3.08                 & 2.86E-06             & 3.08                 & 2.86E-06             & 3.08                 \\
160                  & 3.45E-07             & 3.05                 & 3.48E-07             & 3.04                 & 3.47E-07             & 3.04                 & 3.47E-07             & 3.04                 \\
                     & \multicolumn{1}{l}{} & \multicolumn{1}{l}{} & \multicolumn{1}{l}{} & \multicolumn{1}{l}{} & \multicolumn{1}{l}{} & \multicolumn{1}{l}{} & \multicolumn{1}{l}{} & \multicolumn{1}{l}{} \\
$\mathbb{P}^2$       &                      &                      &                      &                      &                      &                      &                      &                      \\
10                   & 1.09E-04             &                      & 1.09E-04             &                      & 1.09E-04             &                      & 1.09E-04             &                      \\
20                   & 1.91E-06             & 5.84                 & 1.91E-06             & 5.84                 & 1.91E-06             & 5.84                 & 1.91E-06             & 5.84                 \\
40                   & 3.20E-08             & 5.90                 & 3.34E-08             & 5.84                 & 3.33E-08             & 5.84                 & 3.33E-08             & 5.84                 \\
80                   & 1.48E-09             & 4.44                 & 6.31E-10             & 5.73                 & 6.21E-10             & 5.75                 & 6.20E-10             & 5.75                 \\
160                  & 1.47E-09             & 0.01                 & 2.11E-11             & 4.90                 & 1.28E-11             & 5.60                 & 1.28E-11             & 5.60                 \\ \hline
\end{tabular}
	\caption{$L^{\infty}$ errors for the approximation to the variance (top) and the post-processed variance (bottom) for the periodic problem using the discontinuous Galerkin method using a $\mathbb{P}^k$ polynomial approximation in physical space $x$.}
    \label{tbl:linf_var_error}
\end{table}

\begin{table}[ht]
\centering
\begin{tabular}{lcccccccc}
\hline
\multicolumn{9}{l}{Before post-processing}                                                                                                                                                                   \\ \hline
                     & \multicolumn{1}{l}{} & \multicolumn{1}{l}{} & \multicolumn{1}{l}{} & \multicolumn{1}{l}{} & \multicolumn{1}{l}{} & \multicolumn{1}{l}{} & \multicolumn{1}{l}{} & \multicolumn{1}{l}{} \\
                     & $N=5$                &                      & $N=6$                &                      & $N=7$                &                      & $N=8$                &                      \\
Mesh                 & Error $\sigma$       & Order                & Error $\sigma$       & Order                & Error $\sigma$       & Order                & Error $\sigma$       & Order                \\ \hline
                     & \multicolumn{1}{l}{} & \multicolumn{1}{l}{} & \multicolumn{1}{l}{} & \multicolumn{1}{l}{} & \multicolumn{1}{l}{} & \multicolumn{1}{l}{} & \multicolumn{1}{l}{} & \multicolumn{1}{l}{} \\
$\mathbb{P}^1$       &                      &                      &                      &                      &                      &                      &                      &                      \\
10                   & 6.19E-03             &                      & 6.01E-03             &                      & 6.14E-03             &                      & 6.05E-03             &                      \\
20                   & 1.63E-03             & 1.93                 & 1.57E-03             & 1.94                 & 1.61E-03             & 1.93                 & 1.58E-03             & 1.93                 \\
40                   & 4.15E-04             & 1.97                 & 3.99E-04             & 1.97                 & 4.12E-04             & 1.97                 & 4.03E-04             & 1.97                 \\
80                   & 1.05E-04             & 1.99                 & 1.01E-04             & 1.99                 & 1.04E-04             & 1.99                 & 1.01E-04             & 1.99                 \\
160                  & 2.62E-05             & 1.99                 & 2.52E-05             & 1.99                 & 2.61E-05             & 1.99                 & 2.55E-05             & 1.99                 \\
                     & \multicolumn{1}{l}{} & \multicolumn{1}{l}{} & \multicolumn{1}{l}{} & \multicolumn{1}{l}{} & \multicolumn{1}{l}{} & \multicolumn{1}{l}{} & \multicolumn{1}{l}{} & \multicolumn{1}{l}{} \\
$\mathbb{P}^2$       &                      &                      &                      &                      &                      &                      &                      &                      \\
10                   & 5.00E-04             &                      & 4.87E-04             &                      & 4.84E-04             &                      & 4.90E-04             &                      \\
20                   & 5.93E-05             & 3.08                 & 5.66E-05             & 3.11                 & 6.04E-05             & 3.00                 & 5.61E-05             & 3.13                 \\
40                   & 7.24E-06             & 3.03                 & 6.83E-06             & 3.05                 & 7.15E-06             & 3.08                 & 6.92E-06             & 3.02                 \\
80                   & 8.90E-07             & 3.03                 & 8.40E-07             & 3.02                 & 8.81E-07             & 3.02                 & 8.51E-07             & 3.02                 \\
160                  & 1.10E-07             & 3.01                 & 1.04E-07             & 3.01                 & 1.09E-07             & 3.01                 & 1.06E-07             & 3.01                 \\ \hline
\multicolumn{9}{l}{After post-processing}                                                                                                                                                                    \\ \hline
                     & \multicolumn{1}{l}{} & \multicolumn{1}{l}{} & \multicolumn{1}{l}{} & \multicolumn{1}{l}{} & \multicolumn{1}{l}{} & \multicolumn{1}{l}{} & \multicolumn{1}{l}{} & \multicolumn{1}{l}{} \\
\multicolumn{1}{c}{} & $N=5$                &                      & $N=6$                &                      & $N=7$                &                      & $N=8$                &                      \\
Mesh                 & Error $\sigma^*$     & Order                & Error $\sigma^*$     & Order                & Error $\sigma^*$     & Order                & Error $\sigma^*$     & Order                \\ \hline
                     & \multicolumn{1}{l}{} & \multicolumn{1}{l}{} & \multicolumn{1}{l}{} & \multicolumn{1}{l}{} & \multicolumn{1}{l}{} & \multicolumn{1}{l}{} & \multicolumn{1}{l}{} & \multicolumn{1}{l}{} \\
$\mathbb{P}^1$       &                      &                      &                      &                      &                      &                      &                      &                      \\
10                   & 1.22E-03             &                      & 1.22E-03             &                      & 1.22E-03             &                      & 1.22E-03             &                      \\
20                   & 1.27E-04             & 3.26                 & 1.27E-04             & 3.26                 & 1.27E-04             & 3.26                 & 1.27E-04             & 3.26                 \\
40                   & 1.43E-05             & 3.15                 & 1.43E-05             & 3.15                 & 1.43E-05             & 3.15                 & 1.43E-05             & 3.15                 \\
80                   & 1.69E-06             & 3.08                 & 1.69E-06             & 3.08                 & 1.69E-06             & 3.08                 & 1.69E-06             & 3.08                 \\
160                  & 2.05E-07             & 3.04                 & 2.05E-07             & 3.04                 & 2.06E-07             & 3.04                 & 2.05E-07             & 3.04                 \\
                     & \multicolumn{1}{l}{} & \multicolumn{1}{l}{} & \multicolumn{1}{l}{} & \multicolumn{1}{l}{} & \multicolumn{1}{l}{} & \multicolumn{1}{l}{} & \multicolumn{1}{l}{} & \multicolumn{1}{l}{} \\
$\mathbb{P}^2$       &                      &                      &                      &                      &                      &                      &                      &                      \\
10                   & 6.85E-05             &                      & 6.84E-05             &                      & 6.84E-05             &                      & 6.84E-05             &                      \\
20                   & 1.19E-06             & 5.85                 & 1.19E-06             & 5.85                 & 1.19E-06             & 5.85                 & 1.19E-06             & 5.85                 \\
40                   & 2.01E-08             & 5.88                 & 2.06E-08             & 5.85                 & 2.06E-08             & 5.85                 & 2.06E-08             & 5.85                 \\
80                   & 8.79E-10             & 4.52                 & 3.84E-10             & 5.75                 & 3.81E-10             & 5.76                 & 3.80E-10             & 5.76                 \\
160                  & 1.03E-09             & -0.23                & 1.22E-11             & 4.98                 & 7.80E-12             & 5.61                 & 7.77E-12             & 5.61                 \\ \hline
\end{tabular}
\caption{$L^{2}$ errors for the approximation to the variance (top) and the post-processed variance (bottom) for the periodic problem using the discontinuous Galerkin method using a $\mathbb{P}^k$ polynomial approximation in physical space $x$.}
\label{tbl:ltwo_variance_error}
\end{table}

In Figure \ref{fig:error_mean_periodic} we plot the errors of the approximated mean before and after post-processing for $N=5$, $\mathbb{P}^2$ and $20$ elements. We can see that the errors before post-processing are highly oscillatory, and that the post-processing smooths out the error surface and greatly reduces its magnitude. In Figure \ref{fig:error_variation_periodic}, we plot the errors of the approximations for $\sigma$ obtained with the same data as before, again the post-processing technique gets rid of the oscillations and greatly reduces the magnitude of the error.

\begin{figure}[!htbp]
    \centering
    \begin{subfigure}{0.49\textwidth}
        \centering
        \includegraphics[width=\textwidth]{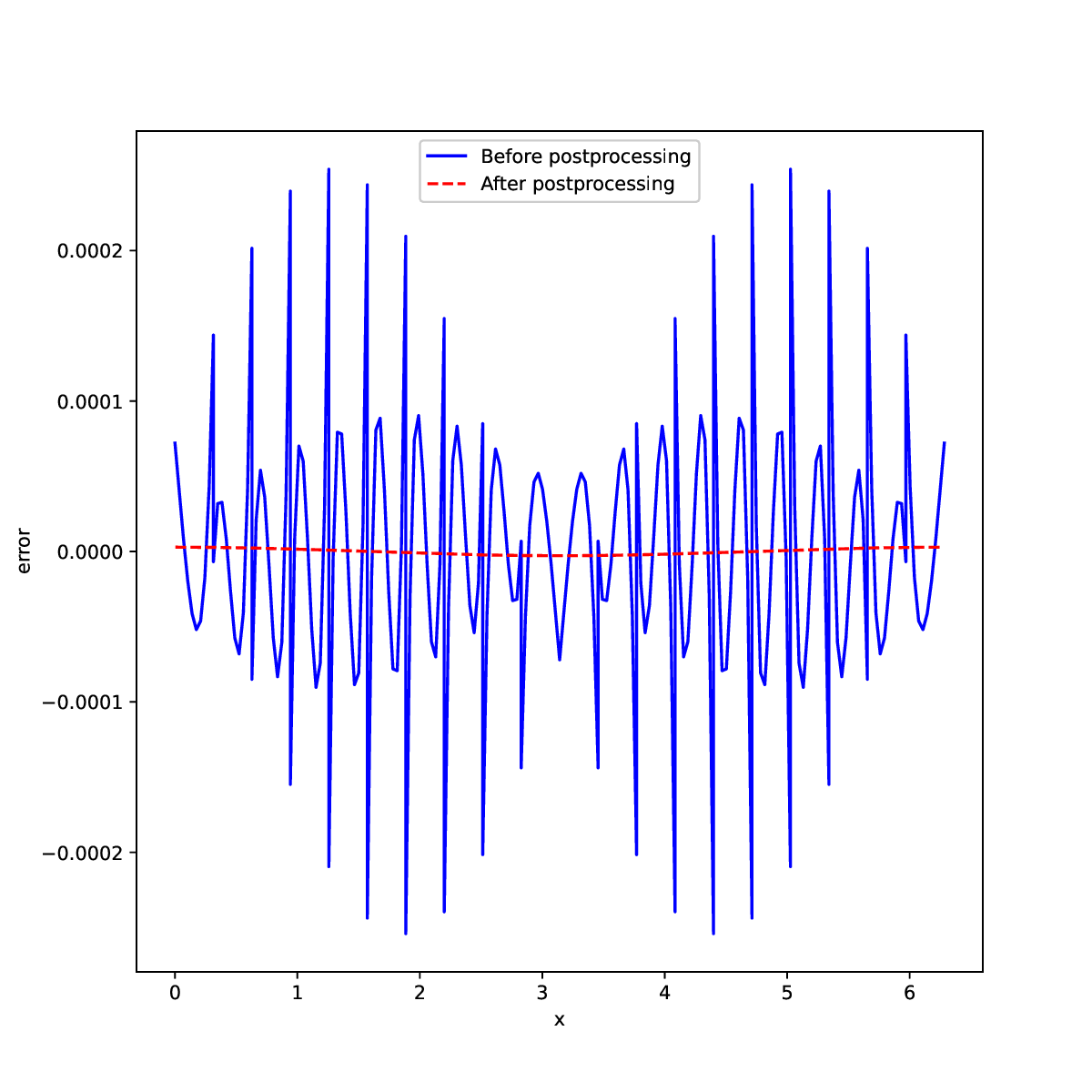}
        \caption{Mean $\mu_u$}
        \label{fig:error_mean_periodic}
    \end{subfigure}
    \hfill
    \begin{subfigure}{0.49\textwidth}
        \centering
        \includegraphics[width=\textwidth]{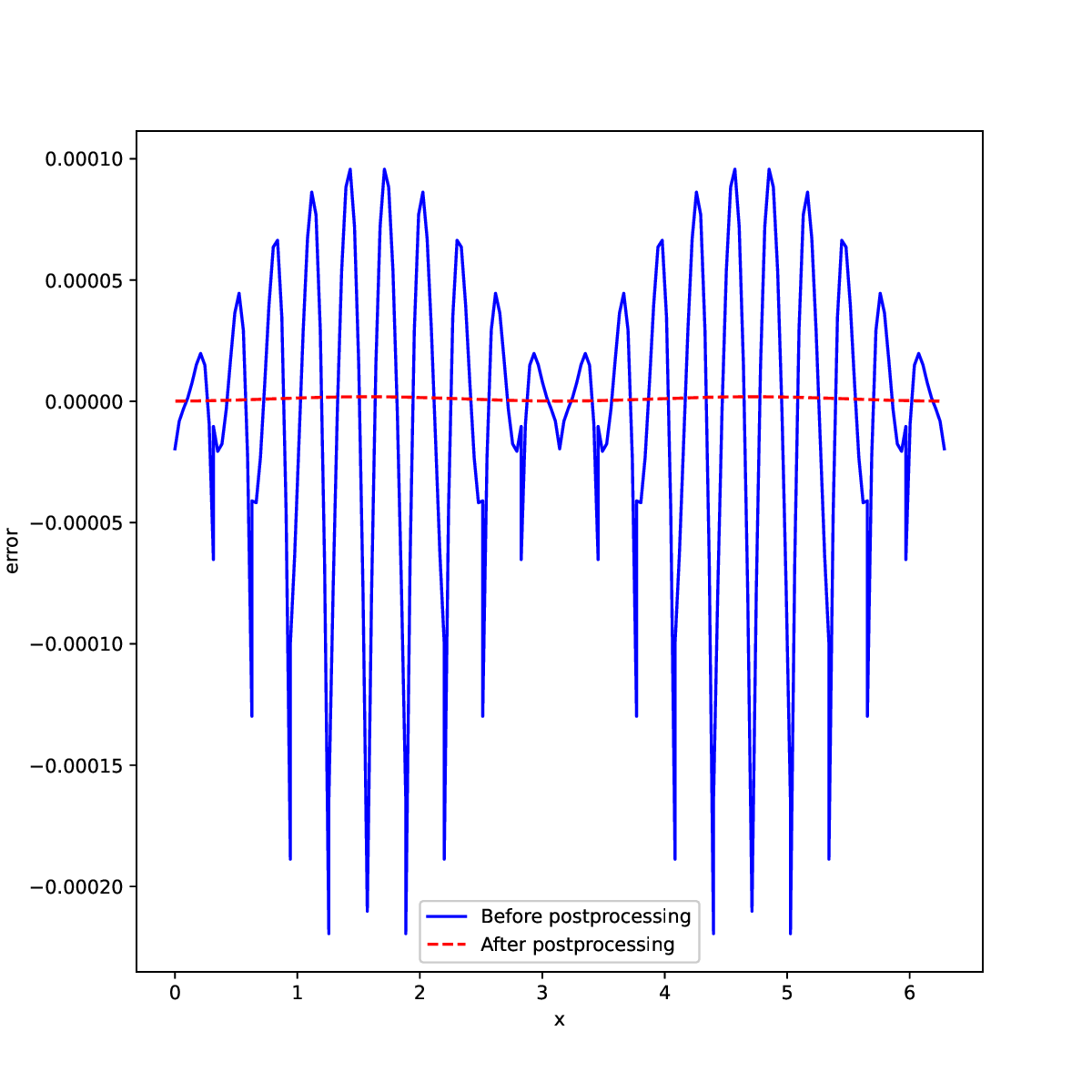}
        \caption{Variance $\sigma^2_u$.}
        \label{fig:error_variation_periodic}
    \end{subfigure}
    \caption{Errors before (solid line) and after post-processing (dashed line) for the mean $\mu_u$ and variance $\sigma^2_u$. $T=1$. $N=5$, $\mathbb{P}^2$ and $20$ elements.}
    \label{fig:error_comparison}
\end{figure}


\subsubsection{Affect of moments and smoothness}

In this section we illustrate the the affect on the approximated coefficients of setting the smoothness, $\ell,$ and number of moments, $r$ in the filter.  In Figure \ref{fig:error_comparison_ell}, the affect of different smoothness parameters for the errors in mean and variance are illustrated.  Here, the number of moments is kept fixed to be $2k,$ while the smoothness, $\ell,$ varies.  Note that the higher the smoothness, the more dissipation is being applied.  We see that overall errors are reduced for both mean and variance provided the filter satisfies $2k$ moments. However, there are more oscillations when $\ell<k+1.$  parameter corr First lest keep fixed the number of moments and play around with the spline degree $\ell$. For  $k=1$.

\begin{figure}[!htbp]
    \centering
    \begin{subfigure}{0.49\textwidth}
        \centering
        \includegraphics[width=0.9\textwidth]{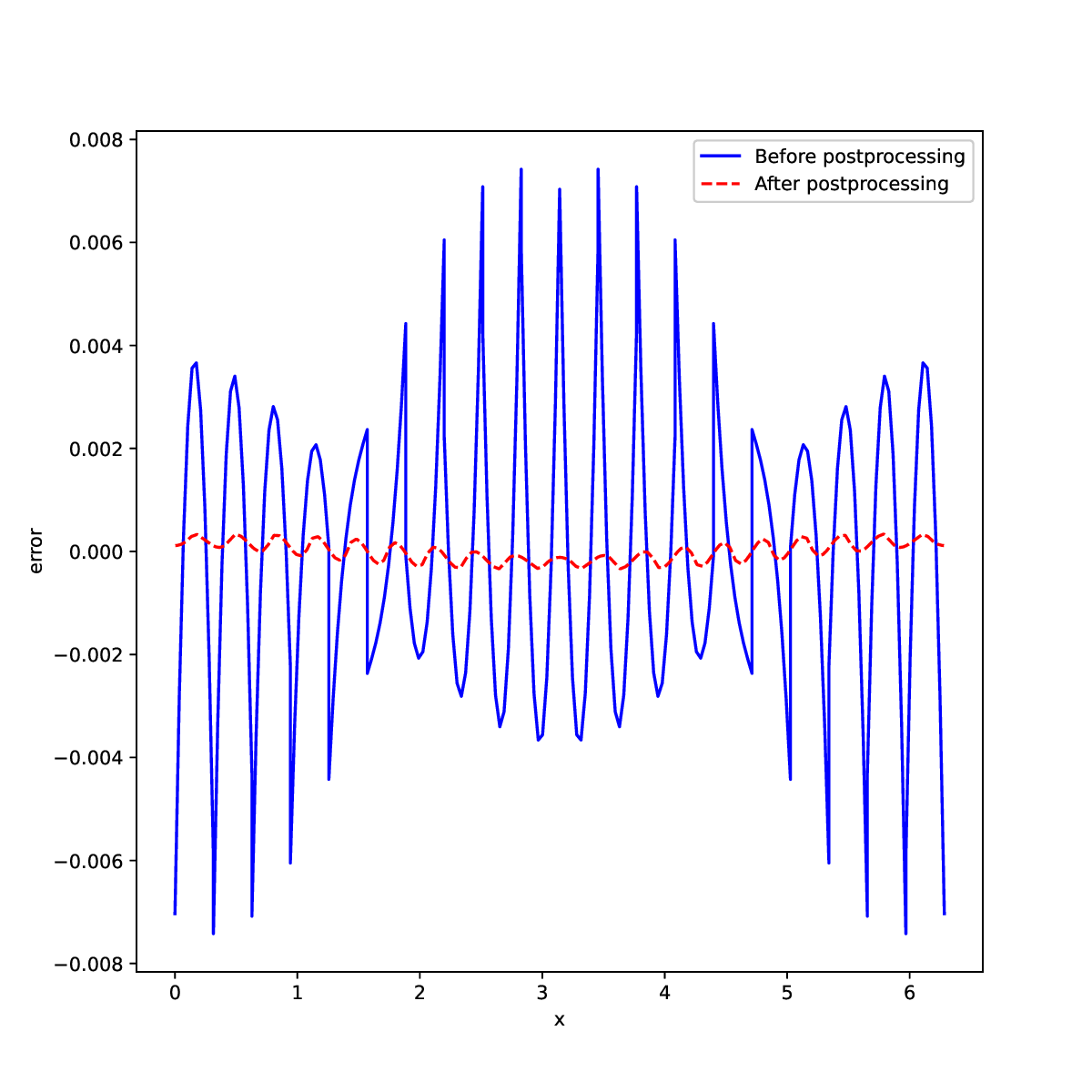}
        \caption{Mean $\mu_u$. $\ell=1$.}
        \label{fig:meanell1momfix}
    \end{subfigure}
    \hfill
    \begin{subfigure}{0.49\textwidth}
        \centering
        \includegraphics[width=0.9\textwidth]{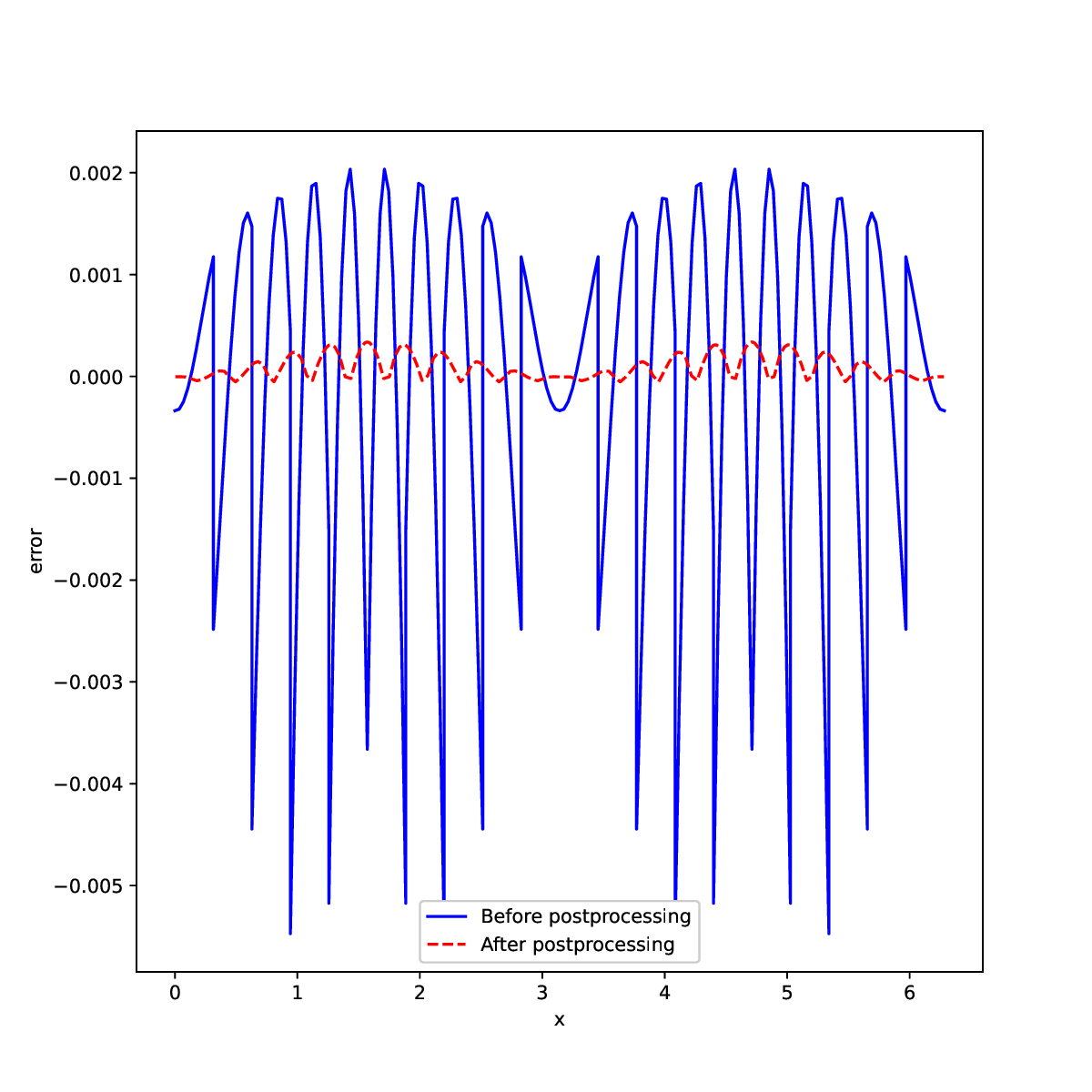}
        \caption{Variance $\sigma^2_u$. $\ell=1$.}
        \label{fig:varell1momfix}
    \end{subfigure}

    \begin{subfigure}{0.49\textwidth}
        \centering
        \includegraphics[width=0.9\textwidth]{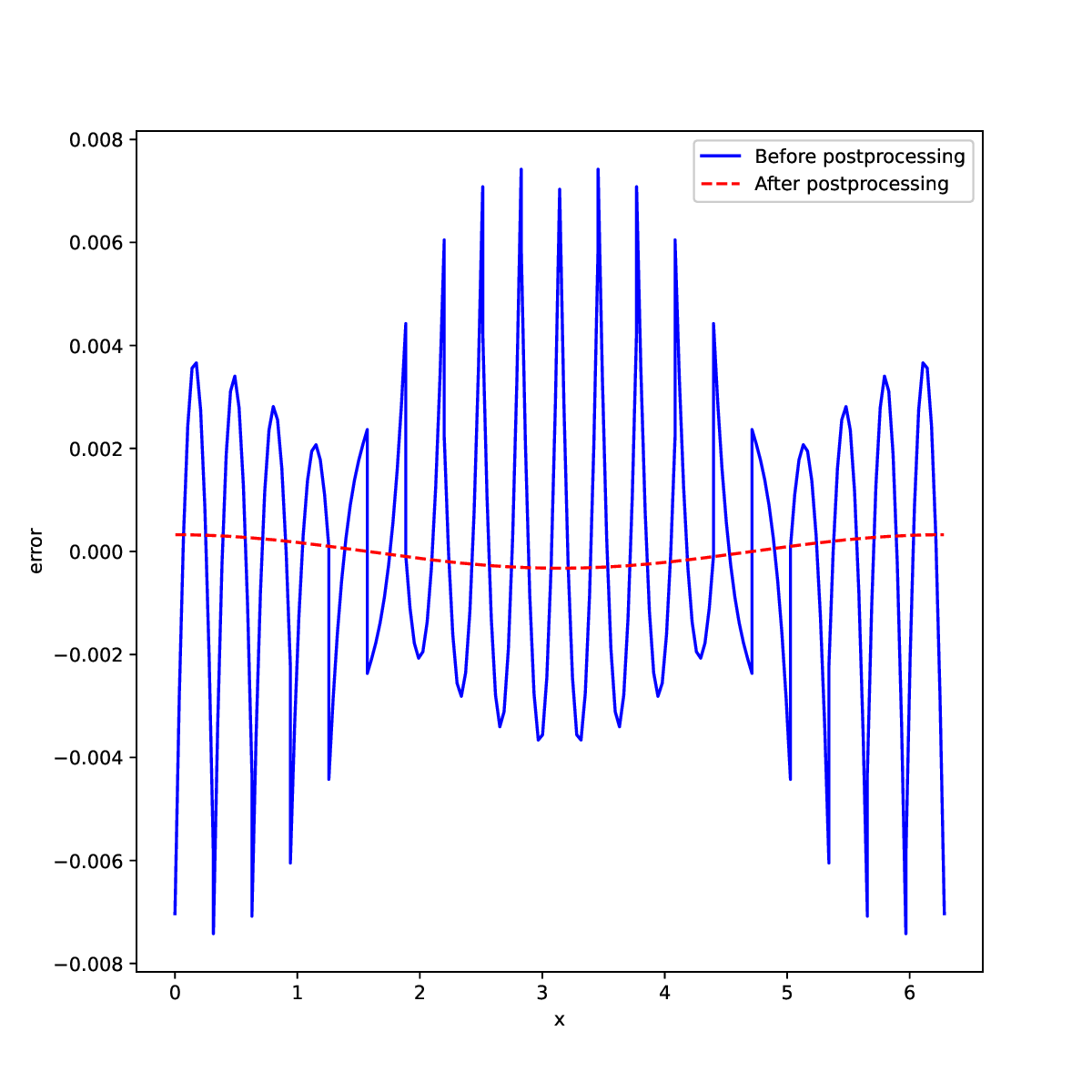}
        \caption{Mean $\mu_u$. $\ell=3$.}
        \label{fig:meanell3momfix}
    \end{subfigure}
    \hfill
    \begin{subfigure}{0.49\textwidth}
        \centering
        \includegraphics[width=0.9\textwidth]{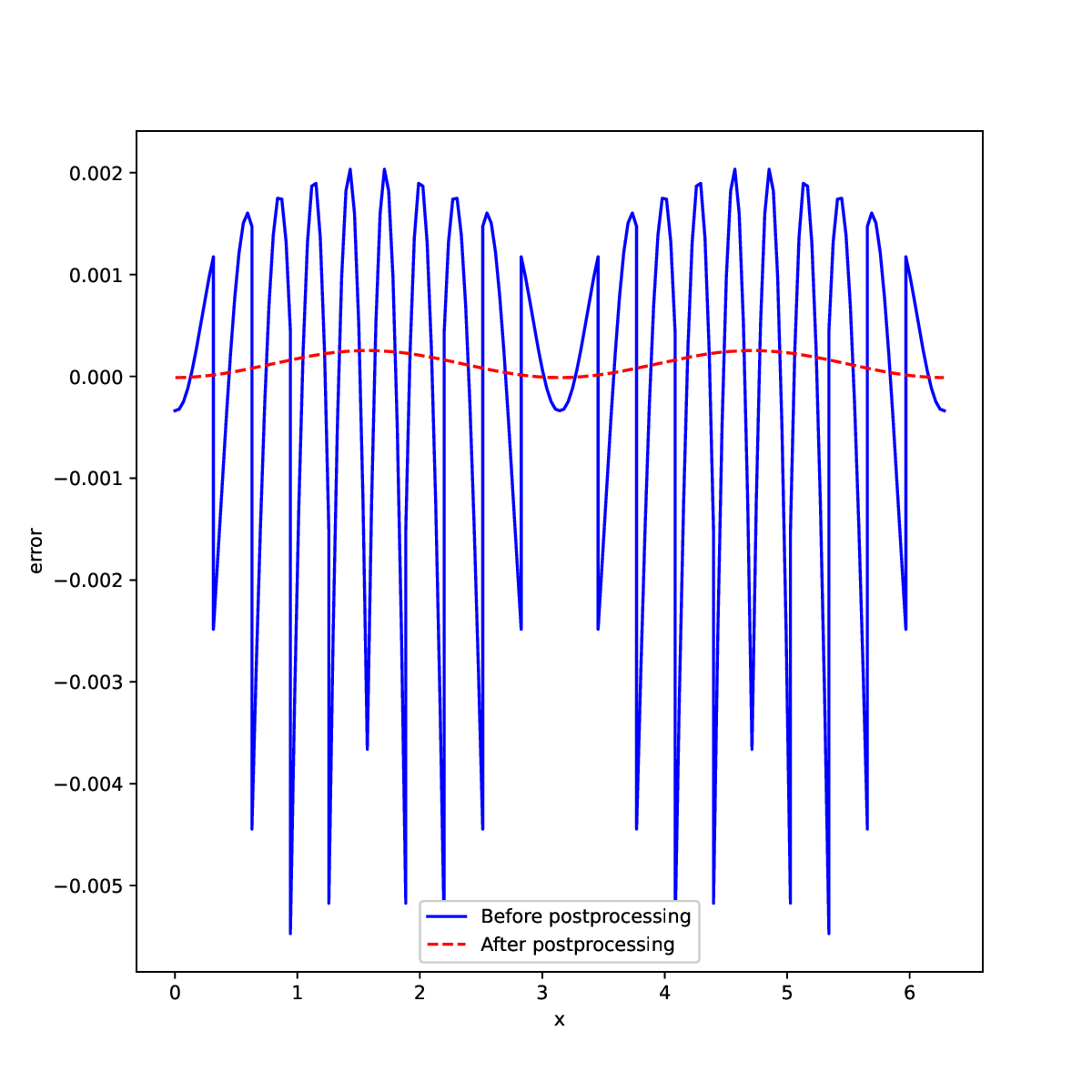}
        \caption{Variance $\sigma^2_u$. $\ell=3$.}
        \label{fig:varell3momfix}
    \end{subfigure}

    \begin{subfigure}{0.49\textwidth}
        \centering
        \includegraphics[width=0.9\textwidth]{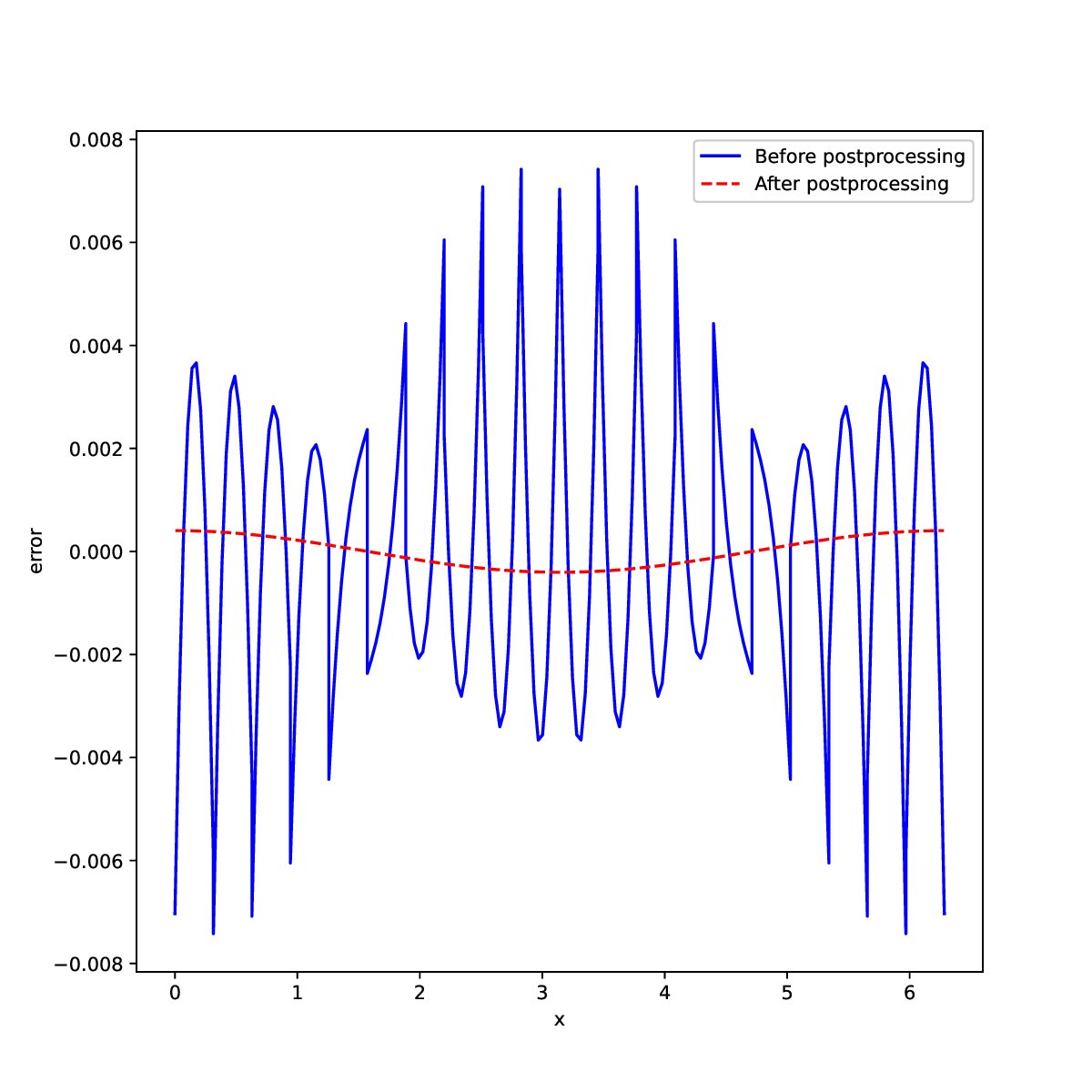}
        \caption{Mean $\mu_u$. $\ell=4$.}
        \label{fig:meanell4momfix}
    \end{subfigure}
    \hfill
    \begin{subfigure}{0.49\textwidth}
        \centering
        \includegraphics[width=0.9\textwidth]{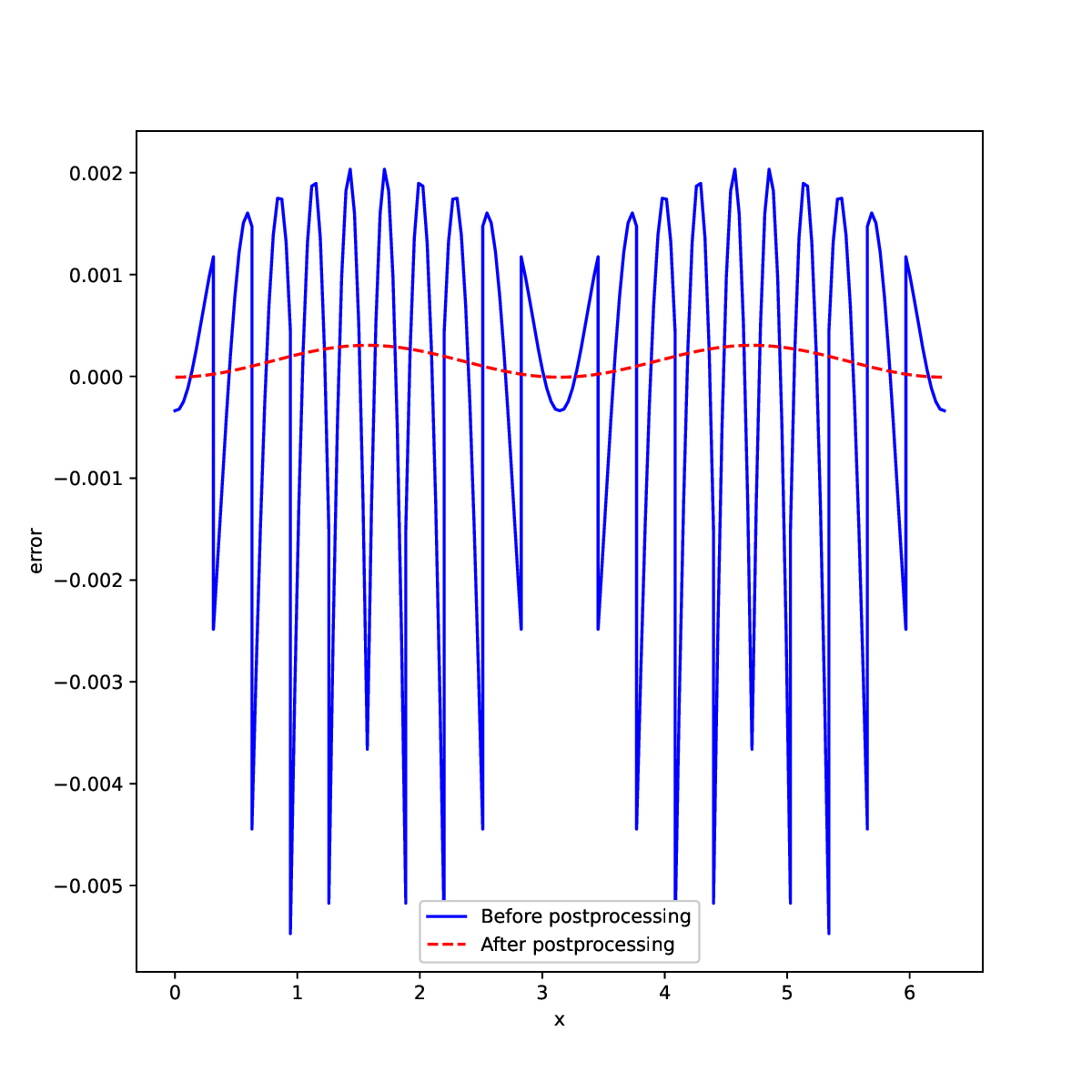}
        \caption{Variance $\sigma^2_u$. $\ell=4$.}
        \label{fig:varell4momfix}
    \end{subfigure}

    \caption{Comparison of errors for the mean and variance before and after post-processing for different values B-Spline degree $\ell$.}
    \label{fig:error_comparison_ell}
\end{figure}

Next, in \Cref{fig:error_comparison_moments}, an illustration of the affect of number of moments is shown.  Here, the  smoothness is $k+1.$  We clearly see that taking fewer than $2k$ moments, the errors for the mean and variance, while smooth, are worse.  As soon as $2k$ moments are used, the errors in the quantities of interest -- mean and variance -- are significantly reduced. 

\begin{figure}[!htbp]
    \centering

    \begin{subfigure}{0.49\textwidth}
        \centering
        \includegraphics[width=0.9\textwidth]{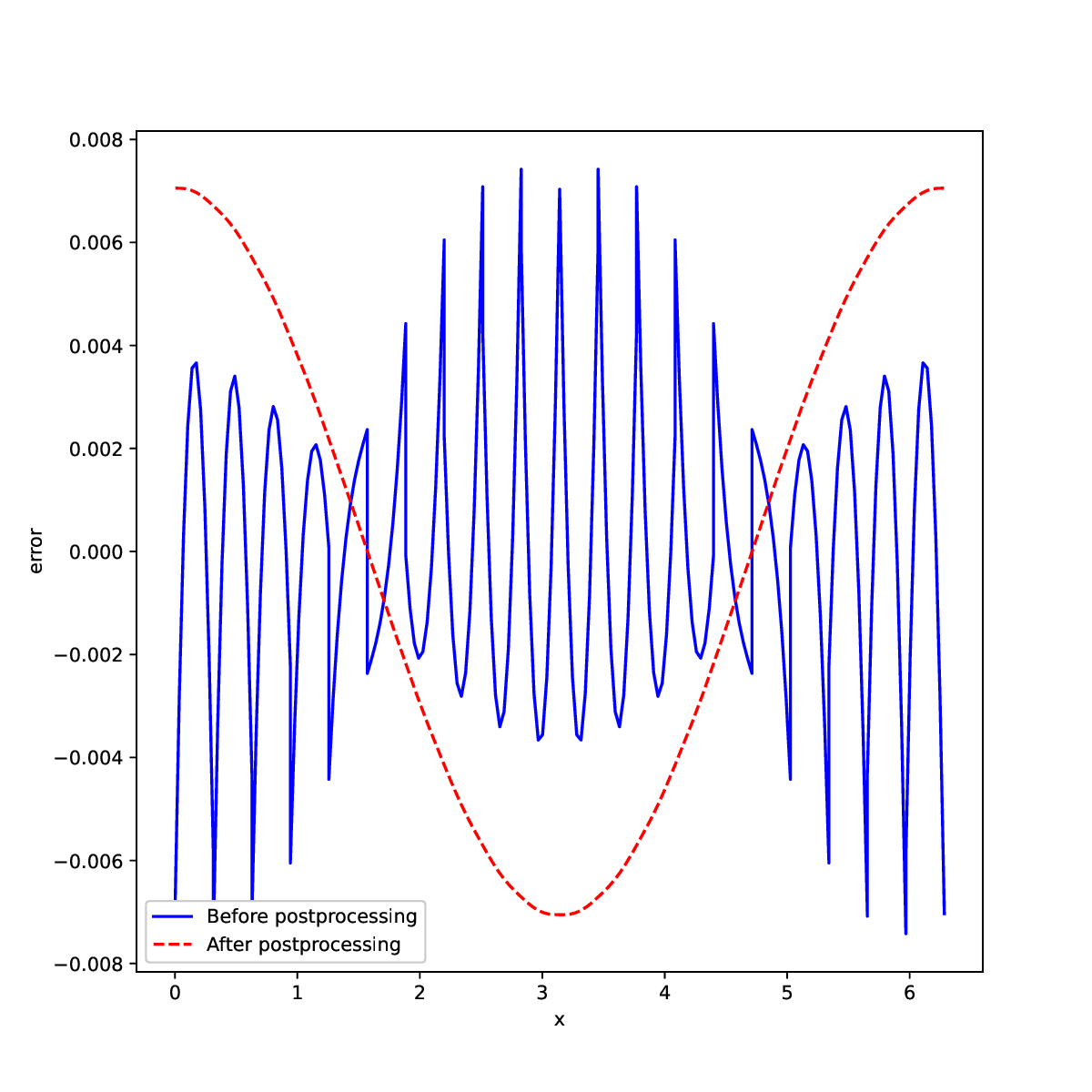}
        \caption{Mean $\mu_u$. $r=0$}
        \label{fig:mean_bs0}
    \end{subfigure}
    \hfill
    \begin{subfigure}{0.49\textwidth}
        \centering
        \includegraphics[width=0.9\textwidth]{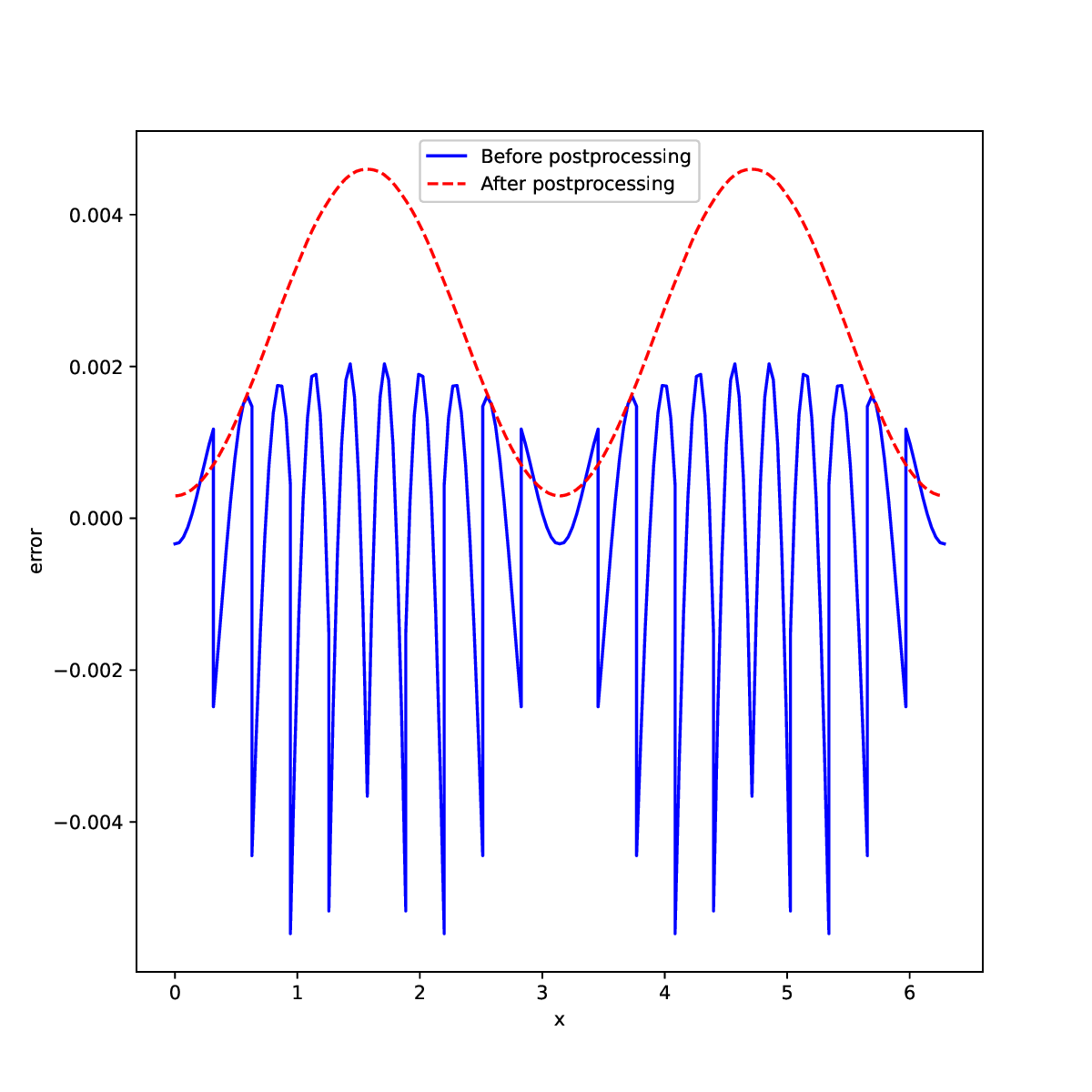}
        \caption{Variance $\sigma^2_u$. $r=0$}
        \label{fig:var_bs0}
    \end{subfigure}

    \begin{subfigure}{0.49\textwidth}
        \centering
        \includegraphics[width=0.9\textwidth]{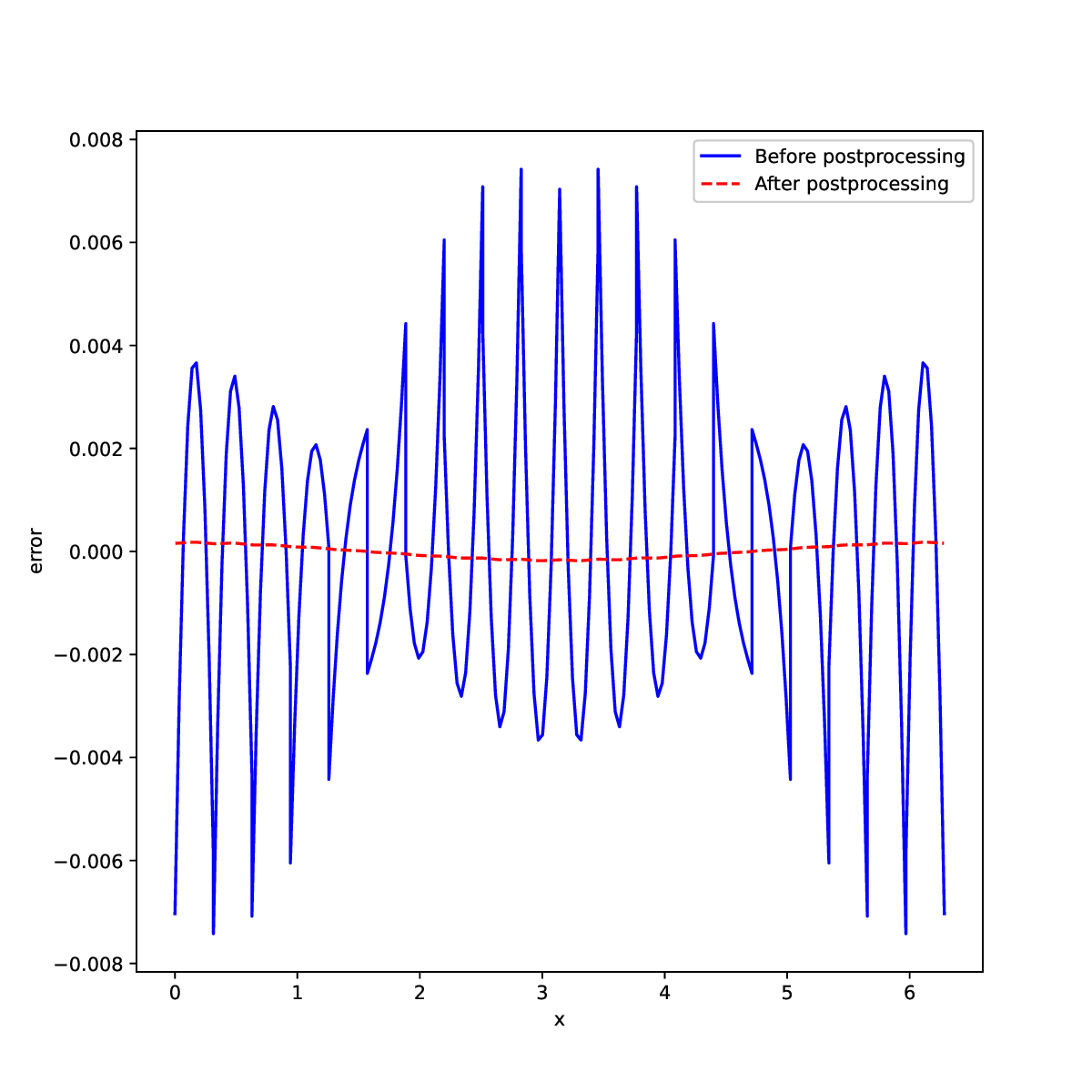}
        \caption{Mean $\mu_u$. $r=2$}
        \label{fig:mean_bs2}
    \end{subfigure}
    \hfill
    \begin{subfigure}{0.49\textwidth}
        \centering
        \includegraphics[width=0.9\textwidth]{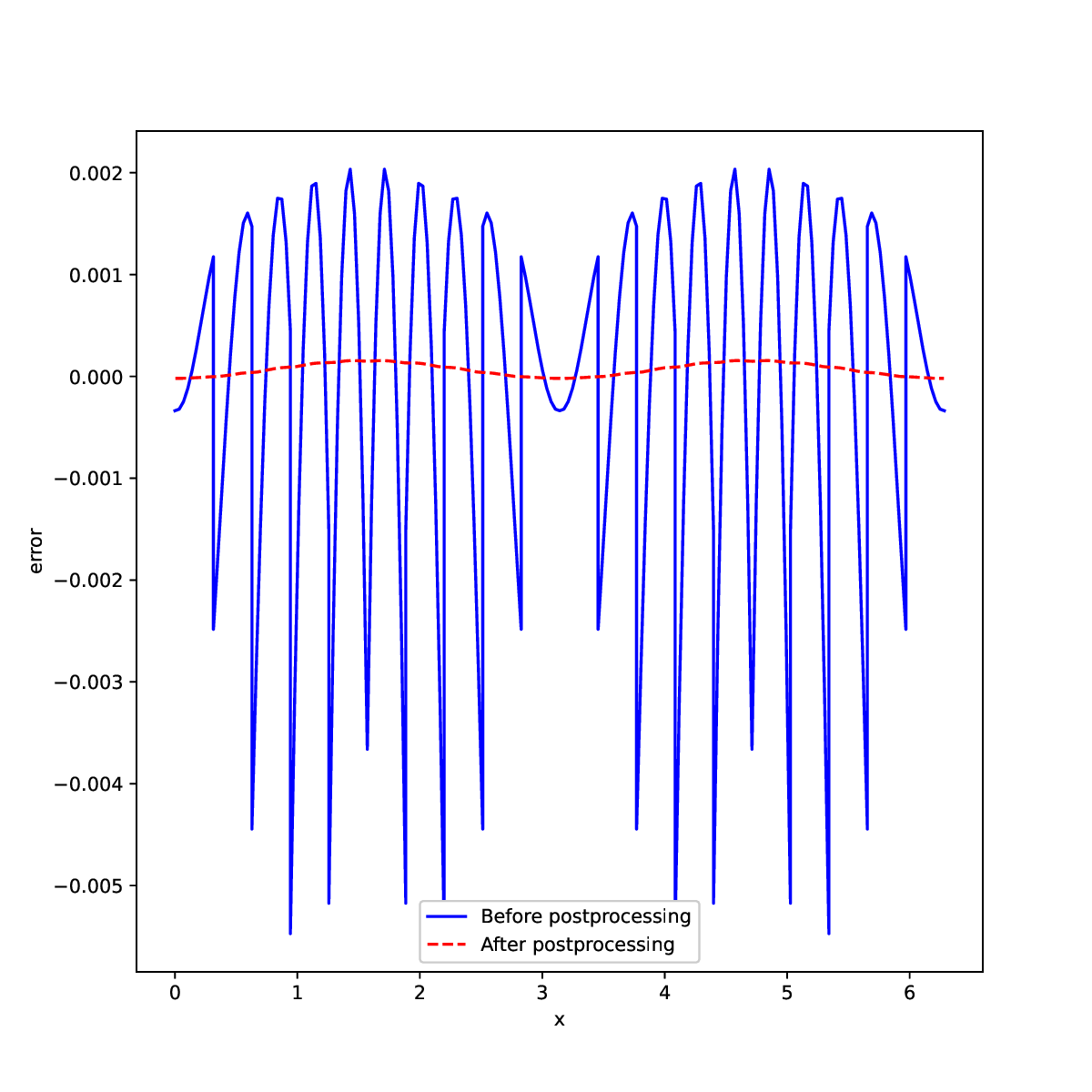}
        \caption{Variance $\sigma^2_u$. $r=2$}
        \label{fig:var_bs2}
    \end{subfigure}

    \begin{subfigure}{0.49\textwidth}
        \centering
        \includegraphics[width=0.9\textwidth]{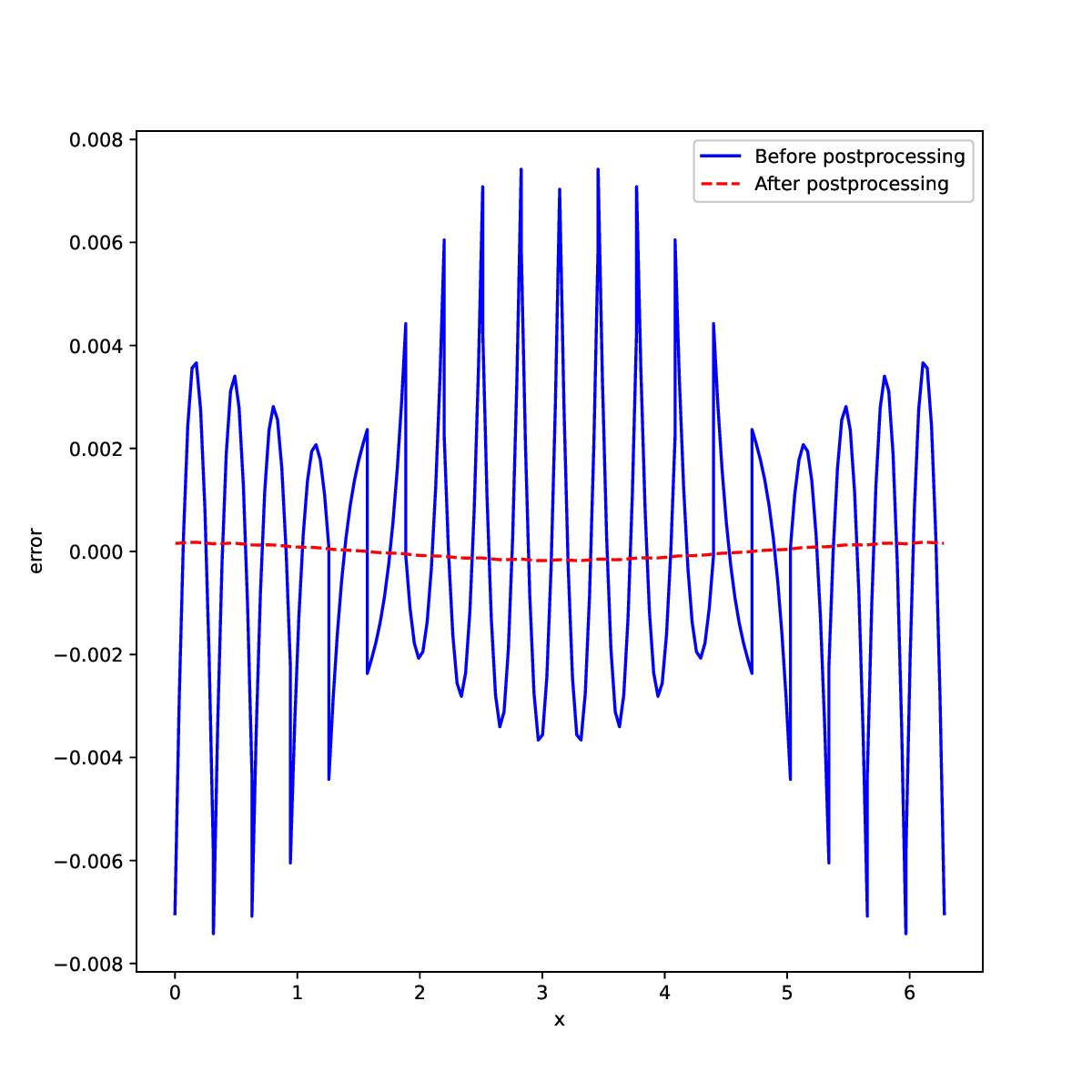}
        \caption{Mean $\mu_u$. $r=6$}
        \label{fig:mean_bs3}
    \end{subfigure}
    \hfill
    \begin{subfigure}{0.49\textwidth}
        \centering
        \includegraphics[width=0.9\textwidth]{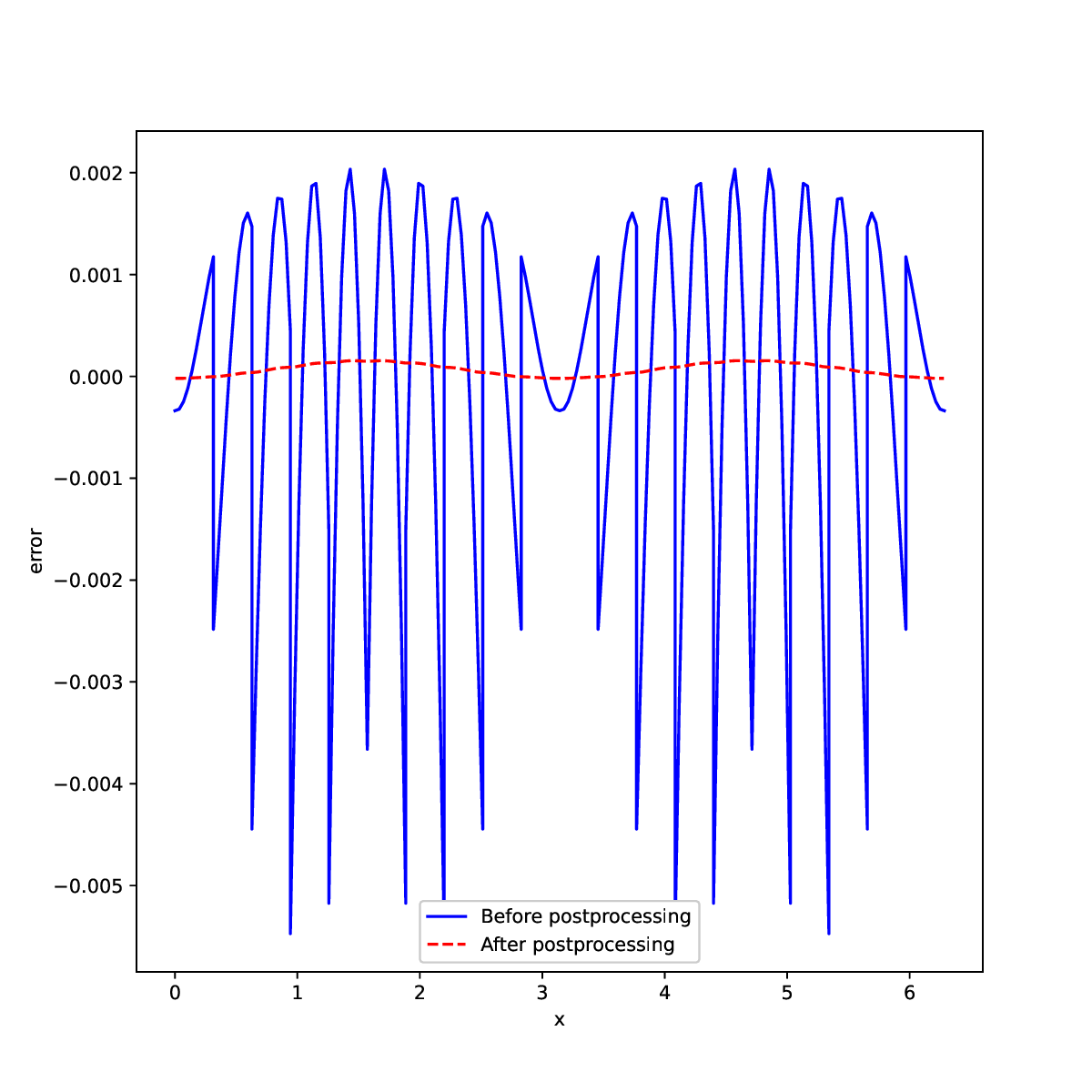}
        \caption{Variance $\sigma^2_u$. $r=6$}
        \label{fig:var_bs3}
    \end{subfigure}

    \caption{Comparison of errors before (solid line) and after post-processing (dashed line) for mean ($\mu_u$) and variance ($\sigma^2_u$) for enforcing different number of moments $r$. $T=1$. $N=5$, $\mathbb{P}^1$ and $20$ elements.}
    \label{fig:error_comparison_moments}
\end{figure}

\section{Conclusion}
The applicability of a SIAC filter applied to approximate coefficients for a generalized polynomial chaos expansion reduces the errors for the  first few moments of the solution to a simple wave equation.  This is exactly what is desired in uncertainty quantification. For the first time (from the knowledge of the authors) error estimates for the DG-gPC are presented for sufficiently smooth solutions both in physical and random space.  In order to establish the error estimate for the SIAC filter, we utilized the information in the negative order norm estimates. Our numerical examples verified the performance of the filter improving and reducing the numerical error and eliminating the noise from the spatial approximation of the mean and variance. Further, we illustrated the affect of the choice of parameters on the quality of the errors.  Hence, this article opens the applicability of SIAC filters to other hyperbolic problems with uncertainty, and other stochastic equations.
\bibliographystyle{plain}
\bibliography{main.bib}
\end{document}